\documentclass[11pt]{article}
\usepackage{graphicx, amsthm, amsfonts , amssymb, amsmath, bbold,xcolor}
\usepackage{mathrsfs}
\usepackage{quiver}
\usepackage[a4paper, top=1.2in, bottom=1.5in, left=1.2in, right=1.1in]{geometry}
\usepackage[style=alphabetic,sorting=nyt,maxbibnames=9,url=false]{biblatex}
\usepackage{hyperref}
\usepackage{comment}
\usepackage{float}
\usepackage{hyperref}

\numberwithin{equation}{section}
\numberwithin{figure}{section}

\newcommand{\T}{\mathbb{T}} 
\DeclareMathOperator{\BBar}{\text{\textbf{\bf Bar}}}
\DeclareMathOperator{\bfT}{\text{\textbf{T}}} 
\DeclareMathOperator{\BV}{\text{\textbf{\bf BV}}} 
\DeclareMathOperator{\B}{\text{\textbf{\bf B}}}

\newcommand{\GL}{\text{GL}}

\newcommand{\rk}{\text{rk}}
\newcommand{\LS}{\rm{LS}}
\newcommand{\iLS}{\rm{iLS}}
\DeclareMathOperator{\TS}{\text{\textbf{\bf TS}}}
\DeclareMathOperator{\dist}{dist}

\newcommand{\norm}[1]{\left \lVert #1 \right \rVert}

\newtheorem{theorem}{Theorem}[section] 

\newtheorem{lemma}[theorem]{Lemma}
\newtheorem{corollary}[theorem]{Corollary}

\newtheorem{proposition}[theorem]{Proposition}
\newtheorem{conjecture}[theorem]{Conjecture}
\newtheorem*{theorem_intro}{Theorem}
\newtheorem*{proposition_intro}{Proposition}

\theoremstyle{definition}
\newtheorem{remark}[theorem]{Remark}
\newtheorem{example}[theorem]{Example}
\newtheorem{definition}[theorem]{Definition}

\definecolor{darkteal}{RGB}{0, 128, 128} 

\title{Discrete signature tensors for persistence landscapes} 
\author{\textsc{\small Vincenzo Galgano, Heather A. Harrington, Daniel Tolosa}}

\AtEndDocument{
    \hrule
    \vspace{2mm}
    \indent {\small (V. Galgano)} \textsc{\footnotesize Max-Planck Institute of Molecular Cell Biology and Genetics, Dresden, Germany; Centre for Systems Biology Dresden, Germany; Faculty of Mathematics, Technische Universität Dresden, Germany} \\
    {\small \textit{E-mail address:} {galgano@mpi-cbg.de }; ORCID: \texttt{0000-0001-8778-575X}}\\
    
    \indent {\small (H. A. Harrington)} \textsc{\footnotesize Max-Planck Institute of Molecular Cell Biology and Genetics, Dresden, Germany; Centre for Systems Biology Dresden, Germany; Faculty of Mathematics and Physics of Life, Technische Universität Dresden, Germany; Mathematical Institute, University of Oxford, Oxford, UK} \\ 
    {\small \textit{E-mail address}: {harrington@mpi-cbg.de }; ORCID: \texttt{0000-0002-1705-7869 }}\\

    \indent {\small (D. Tolosa)} \textsc{\footnotesize School of Mathematical and Statistical Sciences, Arizona State University, Tempe, AZ, United States of America} \\
    {\small \textit{E-mail address}: {dtolosav@asu.edu }; ORCID: \texttt{0000-0003-0586-9808 } }
    
}
\date{}

\begin{document}

\maketitle

\begin{center}
\textit{In memory of Sayan Mukherjee.}\footnote{We dedicate this article to the memory of Sayan Mukherjee. He was a friend and mentor to us and the wider applied topology community. This article reflects Sayan’s vision of bringing statistical and geometric perspectives together to study biological systems. He fostered our curiosity for interdisciplinary mathematics, and his example will continue to guide us as we strive to make the applied topology community more inclusive.}
\end{center}

\begin{abstract}
\noindent Signature tensors of paths are a versatile tool for mathematical data analysis. Recently, they have been applied in the context of vectorisation of persistent homology: after a choice of embedding of barcodes into a space of paths on a vector space, one applies the path signature map, resulting in tensors amenable to statistical and machine-learning methods. Among the different path embeddings, the persistence landscape embedding (PLE) is injective and stable, but PLE is composed with the signature map loses injectivity. Therefore, we address this by proposing a discrete alternative. Persistence landscapes are determined by the time-series of their critical points, of which we compute the discrete signature. We call this composition the {\em discrete landscape feature map} (DLFM), and give results on its injectivity, stability and computability. When studying the injectivity, we complete the proof of a general result due to Diehl, Ebrahimi-Fard and Tapia in the higher-dimensional setting. We showcase the DLFM on a knotted protein dataset, capturing sequence similarity and knot depth with statistical significance. We include an appendix with a preliminary study of Chen signatures of persistence landscapes from the point of view of algebraic geometry. \\
\hfill\break
\noindent {\bf Keywords:} persistent homology, barcodes, persistence landscapes, feature maps, time-series, time warping, path signatures, tensors, knotted proteins, vectorisation. \\
\noindent {\bf MSC 2020 codes:} 55N31, 68T09, 46B85. 
\end{abstract}


{\small \paragraph*{Acknowledgements.} 
Results in Section \ref{subsec:discrete signatures} have benefited from helpful interactions with Joscha Diehl, Kurusch Ebrahimi-Fard and Nikolas Tapia, whom we thank for their kind availability. 
We also thank Agnese Barbensi, Darrick Lee, Vidit Nanda and Leon Renkin for stimulating discussions. 
We are grateful to MPI-MiS Leipzig for hosting the conference \textquotedblleft MEGA 2024\textquotedblright, which informed this research, 
and to MPI-CBG Dresden and CSBD for the excellent working conditions. 
HAH gratefully acknowledges funding from the Royal Society RGF/EA/201074, UF150238 and EPSRC EP/R018472/1, EP/Y028872/1 and EP/Z531224/1. 
VG is member of the Italian national group GNSAGA-INdAM.

\paragraph*{Open Access policy.} The authors have applied a CC-BY public copyright license to any Author Accepted Manuscript (AAM) version arising from this submission. }

\section*{Introduction}

As the scale and complexity of biological data increases, problems such as interpretation, classification and quantification require advanced mathematical methods for investigation. Applying persistent homology from topological data analysis to biological data provides a geometric interpretation of the shapes of data. Here, we propose an alternative approach combining persistent homology and nonlinear algebra. \\
\hfill\break
\indent Persistent homology provides a multiscale geometric descriptor of data that is functorial, stable to perturbations and interpretable: it takes in a filtered simplicial complex that describes the shape of data, and then outputs a persistence module that can be visualised as a {\em barcode}, i.e. a multiset of intervals. Persistent homology has been applied to complex datasets (e.g. neuroscience and biomedicine \cite{gardner2022toroidal,lee2017quantifying,miller2015fruit,benjamin2024multiscale,rabadan2019topological}). In addition to its stability and interpretability, vectorising barcodes is crucial for machine learning integration \cite{ali2023survey}. Some approaches consist of studying the analytical and topological properties of certain continuous paths (e.g. landscape embedding, integrated landscape embedding, Betti curves) associated to landscapes (or barcodes) \cite{Bubenik_2020, Chevyrev_2020, giusti2025topological,giusti2023signatureslipschitzfreespacespaths}. For instance, unlike Betti curves, the (integrated) landscape embedding has been proven to be stable and more meaningful, but still computationally inefficient. \\
\indent In \cite{Chevyrev_2020} the authors proposed a new approach to vectorisation in topological data analysis, leading to {\em feature maps} which associate barcodes to the signature tensors of the aforementioned paths. More precisely, a feature map $\Phi_\bullet$ 
\[
\Phi_\bullet : \BBar \xrightarrow{\iota_\bullet} \BV(V) \xrightarrow{\sigma} \T((V)) 
\]
is given by the composition of a chosen \textit{path embedding} $\iota_\bullet: \BBar \rightarrow \BV(V)$ from barcodes to bounded-variation paths, with the {\em Chen signature map} $\sigma: \BV(V)\rightarrow \T((V))$ associating to a path its signature tensor. {\em Signature tensors} are an algebraic tool firstly introduced by \cite{chen1954iterated} with a rich algebraic and combinatorial structure based on Hopf algebras. A more geometrical approach to signature tensors was firstly proposed in \cite{amendola19}, focusing on related algebraic varieties. Since then, the topic has increasingly attracted the interest of algebraic geometers \cite{amendola2023decomposingtensorspacespath, galuppi2024ranksymmetriessignaturetensors}. A discrete version of signature tensors was introduced by \cite{diehl2020iteratedsumsignature, diehl2020timewarping}, called {\em discrete signature tensor} (or also iterated sum signature in opposition to Chen's iterated integral signature). It associates tensors to time-series, and its combinatorics is ruled by quasi-Hopf algebras. First results on the geometry of discrete signature tensors appear in \cite{bellingeri2025discrete}. 

Our work is inspired by both feature maps from persistent homology and the discrete signature tensors. Among the feature maps proposed in \cite{Chevyrev_2020} the \textit{landscape} and \textit{integrated landscape} path embeddings satisfy desirable stability properties. The main theoretical contribution of this work is a discrete approach to feature maps as an alternative for vectorisation. We introduce the {\em discrete landscape feature map} (DLFM)
\[ \begin{matrix} \Phi_{I} : & \BBar & \stackrel{\iota_{\LS}}{\rightarrow} & \BV(V) & \rightarrow & \TS(V) & \stackrel{\Sigma}{\rightarrow} & \T((\mathscr{V})) \\
& B & \mapsto & \lambda^B & \mapsto & \underline{\bold{x}}^B & \mapsto & \Sigma(\underline{\bold{x}}^B) \end{matrix} \]
which first associates to a barcode $B \in \BBar$ the time-series $\underline{\bold{x}}^B=\{\bold{x}_1,\ldots, \bold{x}_n\} \in \TS(V)$ of critical points of the persistence landscape $\lambda^B \in \BV(V)$ of the barcode, and then the discrete signature tensor $\Sigma(\underline{\bold{x}}^B)\in \T((\mathscr{V}))$ of the time-series (cf. Sec.\ \ref{sec:discrete section} for precise definitions). The map $\BV(V)\rightarrow \TS(V)$ restricted to persistence landscapes is \textquotedblleft canonical\textquotedblright, since persistence landscapes are completely determined by their critical points. Of course, this method may generalise to other path embeddings.\\
\indent We choose the landscape embedding and the discrete signature due to their {\em stability} properties and their {\em discriminative} powers. Indeed, the landscape embedding $\iota_{\LS}$ is stable, unlike other embeddings, and injective (cf.\ \cite[Sec.\ 3.4]{Chevyrev_2020}). However, the discriminative power of $\iota_{\LS}$ gets lost in the feature map $\Phi_{\LS}$, since Chen signature tensors are invariant under tree-like equivalence of paths. We address this by taking a discrete approach, given the discrete nature of the output of persistent homology, and the availability of the software FRUITS (Feature Extraction Using Iterated Sums) proposed in \cite{diehl2024fruits} for efficiently implementing discrete signatures. We show that discrete signature tensors distinguish among tree-like equivalent landscapes, implying that the DLFM $\Phi_I$ is more descriptive than $\Phi_{\LS}$. Alternatively, one may solve the non-injectivity of $\Phi_{\LS}$ by adding a time parameter to $\iota_{\LS}$, but we expect this would have computational limitations. \\
\indent We start our analysis of the discrete feature map $\Phi_I$ by first studying the injectivity of the discrete signature map $\Sigma: \TS(V)\rightarrow \T((\mathscr{V}))$. Such a map is a monoid homomorphism with respect to the shifted-concatenation (cf.\ equation \eqref{def:concatenation time-series}) of time-series, and the tensor product of discrete signature tensors. The discrete signature map is known \cite{diehl2020timewarping} to be invariant under time-warping and time-translation equivalence $\sim_{twt}$. We prove that this equivalence uniquely determines the (monoidal) kernel of $\Sigma$. This result appeared in \cite[Theorem 3.15]{diehl2022tropical} and we fill the gap in their proof from one-dimensional to higher-dimensional time-series.

\begin{theorem_intro}[cf.\ Theorem \ref{thm:uniqueness up to time-warping}]
Any two time-series $\underline{\bold{x}},\underline{\bold{y}} \in \TS(V)$ have the same discrete signature $\Sigma(\underline{\bold{x}}) = \Sigma(\underline{\bold{y}})$ if and only if $\underline{\bold{x}} \sim_{twt} \underline{\bold{y}}$. In particular, the quotient map $\TS(V)/_{\sim_{twt}} \hookrightarrow \T((\mathscr{V}))$ is injective.
\end{theorem_intro}

\noindent We also propose a description of time-warping and time-translation in terms of matrices (cf.\ Remark \ref{rmk:matrix time-warping}).\\
\indent Next, we address the problem of the stability of the discrete feature map $\Phi_{I}$. 
We note that it is not continuous on the whole domain $\BBar$, since the discrete signature map is not invariant under time-series refinements. We overcome this obstacle by restricting to consider a finite set of barcodes $D\subset \BBar$, and consider the union of the critical points of all landscapes arising from $D$. This results in the {\em specialised discrete landscape feature map} $\Phi_{I_D}$. Such an assumption is reasonable considering that most experiments generate finite datasets. 
Moreover, we restrict the image of $\Phi_{I_D}$ to the subspace $\T^{\leq k}((\mathscr{V}))$ of discrete signature tensors truncated up to weight $k$. Again, this restriction is quite natural for applications.\\
\indent The feature map $\Phi_{I_D}$ has desirable continuity properties. Indeed, albeit not being stable in the sense of $1$-Lipschitz, we show that $\Phi_{I_D}$ is uniformly continuous. Here, the metrics involved are the {\em bottleneck distance} on the space of barcodes $\BBar$ (introduced by \cite{cohen2007stability} and further investigated by \cite{bubenik2015statistical}) and the {\em Bombieri-Weyl distance} on the space of discrete signature tensors $\T((\mathscr{V}))$, naturally induced from the standard scalar product on $\mathbb R^d\simeq V$.

\begin{theorem_intro}[cf.\ Theorem \ref{thm:uniform continuity}]
Given $D\subset \BBar$ a finite set of barcodes, the specialised discrete landscape feature map 
\[ \Phi_{I_D}^{\leq k}: \BBar \rightarrow \T^{\leq k}((\mathscr{V})) \]
is uniformly continuous with respect to the bottleneck distance on $\BBar$ and the Bombieri-Weyl distance on $\T^{\leq k}((\mathscr{V}))$.
\end{theorem_intro}

\indent We complement the theoretical study by showcasing the DLFM on a knotted protein dataset previously analysed with persistence landscapes in \cite{benjamin2023homology}. We use the software FRUITS (Feature Extraction Using Iterated Sums) proposed in \cite{diehl2024fruits} for implementing discrete signatures. We show that discrete signatures and structure representative classes (in the sense of amino-acid sequence similarity) are strongly correlated with very high {\em Adjusted Rand Index} (ARI) and {\em Normalised Mutual Information} (NMI) scores. We validate the statistical significance of such correlation by running permutation tests. We also show correlation between discrete signature and knot depth via two statistical methods, and we visualise it by applying Principal Component Analysis (PCA). In our analysis we maintain the labeling of representatives analogous to the one in \cite{benjamin2023homology} to allow comparisons. We explain our analysis in a way intended to be accessible to a broad audience.\\
\hfill\break
\indent Finally, we give a closer look at the classical landscape feature maps and at their images in $\T((V))$. In particular, we focus on the varieties of signature matrices of landscape paths. For this pourpose, as standard in algebraic geometry, we move to the complex field $\mathbb C$ and the projective setting. Our partial description is intended as invitation to algebraic geometers for further study. For instance, since persistence landscape paths are piecewise linear loops, we first describe the signature matrices for loops. 

\begin{proposition_intro}[cf.\ Corollary \ref{cor:loop_variety}]
The projective variety of signature matrices of piecewise linear loops in $\mathbb R^d$ with $m$ maximal segments coincides with the secant variety $\sigma_{\lfloor \frac{m}{2}\rfloor}(Gr(2,d))$ of the Grassmannian of planes in $\mathbb C^d$.
\end{proposition_intro}

\indent This work is organised as follows. Sec.\ \ref{sec:preliminaries} contains the preliminary notions and already known results about signature tensors and persistence feature maps. Sec.\ \ref{sec:discrete section} represents the theoretical core of the manuscript, and contains all constructions and results about the discrete landscape feature map. Sec.\ \ref{sec:application} highlights the utility of the signature-based analysis of biological structures of knotted proteins with respect to both sequence similarity and knot depth. The Appendix \ref{sec:signature for landscapes} contains the discussion of signature tensors of persistence landscape from an algebro-geometric perspective.

\section{Preliminaries}\label{sec:preliminaries}

\subsection{Signature tensors}\label{subsec:prelim_signatures}

Our main references for notions and result appearing in the following are \cite{amendola19, amendola2023decomposingtensorspacespath}. We fix a vector space $V$ over $\mathbb R$ of dimension $d$ and a basis $(e_1,\ldots, e_d)$.

\paragraph{Power series tensor algebra.} 
Given a real vector space $V$, let $\T((V)):= \prod_{k\geq 0}V^{\otimes k}$ be the {\em power series tensor algebra} over $V$, whose elements are sequences of tensors $\mathcal T=\big( T^{(k)} \big)_k$ such that $T^{(k)} \in V^{\otimes k}$, and whose inner product is given by the tensor product $\mathcal T \otimes \mathcal T' = \big(\sum_{i+j=k}T^{(i)} \otimes (T')^{(j)}\big)_k$. Equivalently, by looking at elements as power series $\mathcal T=\sum_{k\geq 0}T^{(k)}$, the inner product is the multiplication of power series. For any $k$ we write $T^{(k)}_{i_1,\ldots, i_k} \in \mathbb R$ for the $(i_1,\ldots, i_k)$-th entry of the tensor $T^{(k)}$ with respect to the fixed basis of $V$, so that $T^{(k)}=\sum_{I \in [n]^k} T^{(k)}_{i_1,\ldots, i_k}e_{i_1}\otimes \ldots \otimes e_{i_k}$. We denote the multiplicative subgroup of $\T((V))$ of power series with degree--$0$ term equal to $1$ by
\begin{align*}
\T_1((V)) & := \{ \mathcal T \in \T((V)) \ | \ T^{(0)}=1 \}  \ .
\end{align*}

\paragraph{Path signature.} Let $\BV([0,1],V)$ be the set of bounded-variation continuous paths $\gamma: [0,1]\rightarrow V$. The {\em path signature} is the map
\begin{equation}\label{eq:signature path}
\begin{matrix}
\sigma : & \BV([0,1],V) & \longrightarrow & \T_1((V)) \\
& \gamma & \mapsto & \sigma(\gamma) = \left(\sigma^{(k)}(\gamma)\right)_k
\end{matrix}\end{equation}
such that $\sigma^{(0)}(\gamma)=1$ by convention, and $\sigma^{(k)}(\gamma):=\int_{\Delta_k}d\gamma^{\otimes k}$ is the Riemann-Stieltjes integral of $d\gamma^{\otimes k}:= d\gamma(t_1)\otimes \ldots \otimes d\gamma(t_k)$ over the $k$--dimensional standard simplex $\Delta_k$. More practically, the coefficient of $e_{i_1}\otimes \ldots \otimes e_{i_k}$ in the tensor $\sigma^{(k)}(\gamma) \in V^{\otimes k}$ is (cf.\ \cite[Secc.\ 2,3]{chen1954iterated})
\begin{equation}\label{eq:def signature}
    \sigma^{(k)}(\gamma)_{i_1,\ldots, i_k} = \int_{0}^1 \int_{0}^{t_k} \cdots \int_{0}^{t_2}\dot\gamma_{i_1}(t_1)\cdots \dot\gamma_{i_k}(t_k) dt_1\cdots dt_k  \ .
    \end{equation}
The sequence $\sigma(\gamma)\in \T((V))$ is the {\em signature tensor} of $\gamma$, and its $k$--th entry $\sigma^{(k)}(\gamma) \in V^{\otimes k}$ is the {\em order-$k$ signature tensor} of $\gamma$. \\
\indent The path signature $\sigma$ is a multiplicative homomorphism with respect to the concatenation of paths in $\BV([0,1],V)$ (which we denote by $\ast$) and the tensor product in $\T_1((V))$. This is a consequence of the Chen relation \cite[Theorem 3.1]{chen1954iterated}
\begin{equation}\label{eq:chen relation}
    \sigma(\gamma_1 \ast \gamma_2)=\sigma(\gamma_1)\otimes \sigma(\gamma_2) \ .
    \end{equation}

\noindent Moreover, $\sigma$ quotients to an injective homomorphism 
\begin{equation}\label{eq:injective up to tree-like}
\sigma: \ \BV([0,1],V)/\sim_{tl} \ \hookrightarrow \ \T((V))
\end{equation}
with respect to the {\em tree-like equivalence}, for whose definition we refer to \cite[Sec.\ 4]{Chevyrev_2020}. However, \cite[Corollary 1.5]{Hambly_2010} shows that one can take $\gamma_1 \sim_{tl} \gamma_2 \iff \sigma(\gamma_1\ast \gamma_2^{-1})=(1,\mathbf{0},\ldots)$ as definition. 

\begin{remark}
The name {\em tree-like} equivalence comes from the fact that in a tree (say, path of segments without loops) one can always iteratively retract leaves up to a point, giving a trivial signature.
\end{remark}

\begin{remark}\label{remark:group invariance}
The natural action of $\GL(V)$ on $V$ induces an entry-wise action on each $V^{\otimes k}$, hence on $\T_1((V))$, as well as on $\BV([0,1],V)$ where $(g \cdot \gamma)(t)= g \cdot (\gamma(t))$. In particular, the path signature map $\sigma$ is $\GL(V)$--equivariant, that is $g \cdot \sigma(\gamma) = \sigma(g \cdot \gamma)$. Finally, signatures are invariant under translations and reparametrisations of paths. We will therefore always consider paths based at the origin, i.e. $\gamma(0)=\bold{0}\in V$.
\end{remark}

\paragraph*{Signatures of PWL paths.} We refer to {\em piecewise linear} (PWL) paths as continuous paths $\gamma: [0,L] \rightarrow V$ which are concatenation of finitely many segments $\gamma= \alpha_1 \ast \ldots \ast \alpha_m$, where $\alpha_j(t)=\bold{a}_j \cdot t$ for certain $\bold{a}_j \in V$: in the expression of $\gamma$, the paths $\alpha_j(t)$ are implicitly shifted so that the concatenation makes sense. The partition $\gamma=\alpha_1\ast \ldots \ast \alpha_m$ is unique if we require that the segments $\alpha_1, \ldots, \alpha_m$ have maximal length possible, or equivalently that $m$ is minimal. 

\begin{definition}\label{def:minimal segment decomp}
We refer to such partition as the \textit{minimal segment decomposition} (MSD) of the path, and we say that the path $\gamma$ is $m$-piecewise linear ($m$--PWL).
\end{definition}

\indent Signatures of PWL paths have a particularly nice expression in terms of their segment decomposition. The $k$--th signature tensor of a segment path $\alpha(t)=\bold{a}\cdot t$ for $\bold{a}\in V$ is always the symmetric rank--$1$ tensor determined by $\bold{a}$, that is \[\sigma(\alpha) = \left(1, L\bold{a}, \frac{L^2}{2}\bold{a}^{\otimes 2}, \ldots, \frac{L^k}{k!}\bold{a}^{\otimes k}, \ldots \right) = \exp(L\bold{a}).\] In fact, segments (up to tree-like equivalence) are the only paths with rank $1$ signature. Then by Chen's relation \eqref{eq:chen relation} the signature tensor of a $m$--PWL path $\gamma= \alpha_1\ast \ldots \ast \alpha_m$ is 
\begin{equation}\label{eq:Chen on PWL paths}
    \sigma(\gamma) = \exp(L\bold{a}_1) \otimes \ldots \otimes \exp(L\bold{a}_m) \ . 
    \end{equation}

\subsection{Persistence feature maps}\label{subsec:preliminaries landscapes}

A \textit{barcode} is a multiset of intervals of the form $B:=\{ [b_1,d_1), ... , [b_n,d_n)\}$,
where $b_i\leq d_i$ are non-negative real numbers called {\em birth} and {\em death}. Barcodes, or equivalently persistence diagrams, are the standard output of persistent homology, the flagship method in topological data analysis. We refer to \cite{Otter2017Roadmap} for a basic introduction to the topic. In this paper, we will assume barcodes to be {\em tame}, that is they have only finitely many intervals, and all intervals are contained in a certain $[0,L]$. \\
\indent From \cite{bubenik2015statistical} we recall that any barcode $B$ defines a family of continuous functions $\lambda_k^B: \mathbb{R}\rightarrow \mathbb{R}$ such that
\begin{equation}\label{eq:def level}
\lambda_k^B(t) := \sup \{  s\geq0 \,\vert\, \beta^{t-s,t+s} \geq k   \},
\end{equation}
where $\beta^{t-s,t+s}$ is the number of intervals in $B$ that contain $[t-s,t+s]$, and where one sets $\lambda_k^B(t) = 0$ whenever the supremum is taken over the empty set (cf.\ \cite[Definition 3.2]{Chevyrev_2020}). The function $\lambda_k^B(t)$ is called the \textit{$k$--th landscape} of the barcode $B$ \cite{Bubenik_2020_landscapes}, or {\em $k$--th level} of $B$.

\begin{remark} 
A graphical construction of the landscape of a barcode is as follows. Let $B$ be a barcode with bars $(B_1,\ldots ,B_n)$. For any $i \in [n]$, consider the birth-death point $P_i:=(b_i,d_i)\in \mathbb R^2$ such that $B_i$ is supported on $[b_i,d_i]$. By definition, each birth-death point lies above the diagonal in $\mathbb R^2$, and defines a right triangle whose hypotenus lies on the diagonal and whose catheti meet in the point itself. After isometries of the plane, one consider all these hypotenuses lying on the interval $[0,L]$ along the $x$--axis. The persistence landscape of $B$ is the union of the sides of such triangles. 
\end{remark}

\noindent By construction, levels of a landscape $\lambda^B$ are PWL paths $\lambda_j^B:[0,L]\rightarrow \mathbb R$, such that 
\begin{itemize}
\item they are loops: $\lambda_j^B(0)=\lambda_j^B(L)=0$;
\item higher levels dominate lower levels: $\lambda_i^B(t)\geq \lambda_j^B(t)$ for any $t \in [0,L]$ and $i \leq j$;
\item for any $j \in [d]$ the segments of $\lambda_j^B$ have slope in $\{0, \pm 1\}$.
\end{itemize}

\noindent We write $\lambda^B=(\lambda_1, \ldots, \lambda_d)$ by non specifying $B$ at each level when it is clear.

\begin{figure}[H]
\begin{center}
\begin{tikzpicture}[scale=1, every node/.style={font=\small}]

    \draw[thick, black] (0, 0) -- (2.5, 2.5) -- (4,1) -- (5,2)-- (5.75 ,1.25) -- (7,2.5) -- (9.5,0);

    \draw[dashed,thick, blue] (0, 0) -- (3,0) -- (4,1) -- (4.75,0.25) -- (5.75,1.25) -- (7,0) -- (9.5,0);

    \draw[dotted,thick, red] (0, 0) -- (4.5, 0) -- (4.75, 0.25) -- (5,0) -- (9.5 , 0);



\end{tikzpicture}
\caption{{\footnotesize Landscape with $3$ levels: $\lambda_1$ in black, $\lambda_2$ in blue, $\lambda_3$ in red.}}
\end{center}
\end{figure}

\indent In the need of vectorising barcodes, several ways of associating paths in a vector space to barcodes have been introduced under the name of {\em embeddings}. Moreover, in order to obtain a more informative vectorisation, in \cite{Chevyrev_2020} the authors composed such embeddings with the Chen signature map \eqref{eq:signature path}: an embedding $\iota_\bullet : \BBar \rightarrow \BV([0,L],\mathbb R^d)$ defines the {\em feature map}
\[ \Phi_\bullet : \BBar \stackrel{\iota_\bullet}{\longrightarrow} \BV([0,L],\mathbb R^d) \stackrel{\sigma}{\longrightarrow} \T_1((\mathbb R^d)) \ . \]
\noindent Note that the feature map $\Phi_\bullet$ quotients to an injective map under the tree-like equivalence. In the following we focus on two feature maps, but we refer to \cite{Chevyrev_2020} for a more complete overview.\\
\indent The \textit{landscape embedding} is the map associating to each barcode $B$ (with at most $d$ levels) a path $\lambda^B$ in $\mathbb R^d$ whose $k$--entry is the $k$--th level of $B$, namely
\begin{equation}\label{eq:landscape embedding}
\begin{matrix}
    \iota_{\LS} : & \BBar & \longrightarrow & \BV([0,L],\mathbb R^d) \\
    & B & \mapsto & \lambda^B(t) & = \scriptsize{\begin{bmatrix} \lambda_1^B(t) \\ \vdots \\ \lambda_d^B(t) \end{bmatrix}}
    \end{matrix} \ .
    \end{equation}

\noindent Such vectorisation is not as descriptive as one would desire. The tree-like equivalence of paths makes impossible to distinguish large classes of landscapes. For example, the empty barcode, a singleton barcode $\{ [b,d) \}$ and any barcode $\{ [b_1,d_1),...,[b_n,d_n) \}$ such that $b_i > d_{i-1}$, $i=2,...,n$, define tree-like paths. Hence, the feature map $\Phi_{\LS}$ maps each of them to the trivial signature $(1,\bold{0}, \ldots)$.

\begin{figure}[H]
    \centering
    \includegraphics[width=0.25\linewidth]{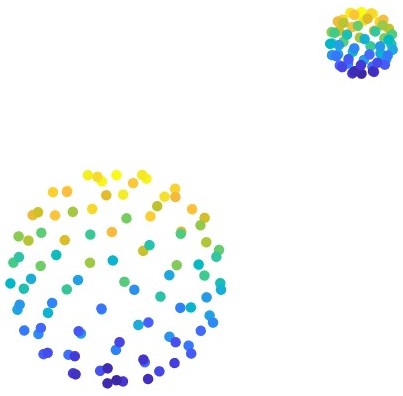}%
    \hspace{1.5cm}
    \begin{tikzpicture}[scale=1]
    \draw[thick, blue] (0,0) -- (1,1) -- (2,0) -- (3,0) -- (5,2) -- (7,0) ;
    \end{tikzpicture}
    \caption{{\footnotesize Uniform sampling, called {\em Fibonacci sampling}, of two spheres (left) has a one-dimensional associated persistence landscape (right) with trivial Chen signature.}}
    \label{fig:two_spheres}
\end{figure}

A more descriptive embedding is given by a cumulative version of $\iota_{\LS}$, which distinguishes more features of barcodes at the cost of an additional integration step. The {\em integrated landscape embedding} is the map
\begin{equation}\label{eq:integrated landscape embedding}
\begin{matrix}
    \iota_{\iLS} : & \BBar & \longrightarrow & \BV([0,L],\mathbb R^d) \\
    & B & \mapsto & F^{\lambda^B}(t) & := \small{\begin{bmatrix} \int_0^t\lambda_1^B(s)ds \\ \vdots \\ \int_0^t\lambda_d^B(s)ds \end{bmatrix}}
    \end{matrix} \ .
    \end{equation}
The key property of the resulting feature map $\Phi_{i\LS}$ is its stability (under perturbations of the barcodes). However, the extra integration step makes it computationally inefficient.\\
\indent Some initial work has been done to describe invariants of signature tensors \cite{galuppi2024ranksymmetriessignaturetensors,amendola2023decomposingtensorspacespath} and it would be interesting to relate these invariants to the topology of barcodes via the feature maps. For instance, the only rank--$1$ signature tensors are signatures of segments \cite[Corollary 7.3]{amendola2023decomposingtensorspacespath}. In the case of the landscape feature map $\iota_{\LS}$, the latter result leads to the following.

\begin{proposition}\label{prop:landscape_rank1}
    Given a barcode $B$, the signature tensor $\Phi_{\LS}(B)=\sigma(\lambda^B)$ has rank $1$ ---i.e. for any $k\geq 0$ the $k$--th signature $\sigma^{(k)}(\lambda^B)$ has tensor-rank $1$ in $V^{\otimes k}$--- if and only if the barcode $B$ is trivial ---i.e. it consists of diagonal birth-death points $[b_i,b_i]$.     
\end{proposition}
\begin{proof}
 From \cite{amendola2023decomposingtensorspacespath}, we know that the signature $\Phi_{\LS}(B)$ has rank $1$ if and only if $\lambda^B$ is a segment. But $\lambda^B$ is a segment if and only if all birth-death pairs of the barcode lie along the diagonal, that is they are of the form $[b_i,b_i]$.
\end{proof}

\section{Discrete feature maps}\label{sec:discrete section}

A discrete version of signature tensors, replacing paths with time-series, was introduced in \cite{diehl2020iteratedsumsignature,diehl2020timewarping}, and first glimpses of the related geometry appear in \cite{bellingeri2025discrete}. Before applying such discrete approach to persistent homology, first we recall definitions and propose a matrix description of the equivalence relation for time-warping and time-translation. Finally, we prove that discrete signatures recover time-series {\em  uniquely up to} time-warping and time-translation, fixing a gap in the proof of \cite[Theorem 3.15]{diehl2022tropical}.

\subsection{Discrete signatures of time-series}\label{subsec:discrete signatures}

Given the polynomial ring $\mathbb R [\bold{Z}]=\mathbb R[Z_1,\ldots, Z_d]$, consider its quotient by the constant polynomials
\[ \mathscr{V}_d = {\mathbb{R}[Z_1,\ldots, Z_d]}/{\mathbb R} \ . \]
As (infinite dimensional) vector subspace, $\mathscr{V}_d$ is generated by all possible monomials in the $d$ variables. We consider the tensor algebra $\T(\mathscr{V}_d)$ and the power series tensor algebra $\T((\mathscr{V}_d))$: note that $\T((\mathscr{V}_d))$ can be identified to the dual space $\T(\mathscr{V}_d)^\vee$.\\
\indent Let $\underline{\bold{x}}=(\bold{x}_1,\ldots , \bold{x}_n)$ be a time-series of $n\geq 1$ vectors (we say, length $n$) in $\mathbb{R}^d$. We denote the set of all time-series (of any positive length) in $\mathbb{R}^d$ by $\TS(\mathbb R^d)$. The {\em discrete signature map} is the map

\begin{equation}\label{eq:discrete signature map}
\begin{matrix}
\Sigma : & \TS(\mathbb R^d) & \longrightarrow & \T((\mathscr{V}_d))\\
& \underline{\bold{x}} & \mapsto & \Sigma(\underline{\bold{x}})
\end{matrix}
\end{equation}
such that, for any time-series $\underline{\bold{x}}=(\bold{x}_1,\ldots, \bold{x}_n)$ of length $n$, the {\em discrete signature} $\Sigma(\underline{\bold{x}})$ is the element of the tensor algebra $\T((\mathscr{V}_d))$ defined as follows: for any $\ell \geq 0$ and any family of monomials $p_1,...,p_\ell \in \mathscr{V}_d$, the coefficient of $p_1\otimes \ldots \otimes p_\ell$ in $\Sigma(\underline{\bold{x}})$ is

\begin{equation}\label{eq:discrete signature}
    \langle \Sigma(\underline{\bold{x}}), p_1 \otimes \cdots \otimes p_\ell \rangle = \begin{cases}
    \sum\limits_{1\leq i_1<...<i_\ell\lneq n} p_1(\bold{x}_{i_1+1}-\bold{x}_{i_1}) \cdots p_\ell(\bold{x}_{i_\ell+1}-\bold{x}_{i_\ell}) & \ell \leq n - 1 \\
    0 & \ell \geq n
    \end{cases}
    \end{equation}

\noindent and $\langle \Sigma(\underline{\bold{x}}), \epsilon \rangle = 1$ where $\epsilon$ corresponds to the empty word. The above expression is a generalisation of the notion of quasi-symmetric polynomials \cite[Remark 3.5]{diehl2020timewarping}. For simplicity, we fix the notation
\[ C_{p_1\otimes \ldots \otimes p_\ell}^{\underline{\bold{x}}} \ := \ \langle \Sigma(\underline{\bold{x}}), p_1\otimes \ldots \otimes p_\ell\rangle \ \in \mathbb R \ . \]
\noindent The {\em weight} of an element $p_1\otimes \ldots \otimes p_\ell$ is the sum of the degrees of the monomials $p_1,\ldots,p_\ell$, that is 
\begin{equation}\label{eq:weight}
    weight(p_1\otimes \ldots \otimes p_\ell) := \deg(p_1)+\ldots + \deg(p_\ell) \ . 
    \end{equation}
\noindent The truncation {\em up to weight $h$} of the discrete signature $\Sigma(\underline{\bold{x}})$ is 
\[ \Sigma^{\leq h}(\underline{\bold{x}}) := \sum_{weight(p_1\otimes \ldots \otimes p_\ell)\leq h} C_{p_1\otimes \ldots \otimes p_\ell}^{\underline{\bold{x}}}p_1\otimes \ldots \otimes p_\ell \ . \]
\indent An important feature of the (standard) signature map is being invariant under time alteration (i.e. reparametrisation of the curve). By definition, the discrete signature carries a similar property, known as {\em time-warping invariance} \cite{diehl2020timewarping}: it is invariant under {\em consecutive} time-point repetitions, that is
\[ \Sigma((\bold{x}_1, \ldots, \bold{x}_{j-1}, \bold{x}_j, \bold{x}_j , \bold{x}_{j+1}, \ldots, \bold{x}_n)) = \Sigma((\bold{x}_1, \ldots, \bold{x}_{j-1}, \bold{x}_j, \bold{x}_{j+1}, \ldots, \bold{x}_n)) \ . \]
\noindent We say that a time-series is {\em time-warping reduced} (TWR) if it does not present consecutive repetitions. It is clear from the definition that any time-series has a unique associated TWR series so that their signatures coincide.\\
\indent As the signature path \eqref{eq:signature path} is injective modulo tree-like equivalence of paths, it is quite natural to ask under which equivalence relation the discrete signature \eqref{eq:discrete signature map} is injective as well. From the above, we know that time-warping is an invariant property. Moreover, again by definition \eqref{eq:discrete signature}, the discrete signature is also invariant under {\em time-translation}:
\[ \Sigma((\bold{x}_1, \ldots, \bold{x}_n)) = \Sigma((\bold{x}_1 - \bold{y}, \bold{x}_2-\bold{y}, \ldots, \bold{x}_n - \bold{y})) \ \ \ \ \ , \ \forall \bold{y}\in \mathbb R^d \ . \]

\indent Let $\sim_{twt}$ be the equivalence relation on $\TS(\mathbb R^d)$ defined by time-warping and time-translation. Let $\{\bold{0}\}$ be the time-series of length $1$ given by the zero vector in $\mathbb R^d$. Then 
\[ \underline{\bold{y}} \sim_{twt} \{\bold{0}\} \ \iff \ \underline{\bold{y}} = \{ \bold{y}_1, \ldots, \bold{y}_1\} \ . \]

\noindent Clearly, all the above time-series have trivial discrete signature 

\begin{equation}\label{eq:monoid homomorphism} 
    \Sigma(\{\bold{0}\})=\epsilon \ . 
    \end{equation}

\begin{remark}[Matrix description of $\sim_{twt}$]\label{rmk:matrix time-warping}
To a time-series $\underline{\bold{x}}=\{\bold{x}_1,\ldots, \bold{x}_n\}\in \TS(\mathbb R^d)$ we associate the matrix $\bold{X}=(x_{ij})=\left[ \bold{x}_1 | \cdots | \bold{x}_n \right]\in \mathbb R^{d\times n}$. The time-translation $\underline{\bold{x}} \mapsto \{ \bold{0}, \bold{x}_2 - \bold{x}_1, \ldots, \bold{x}_n - \bold{x}_1\}$ corresponds to the right-multiplication below:
\begin{equation}\label{eq:matrix time-translation}
\bold{X}\cdot {\scriptsize{\begin{bmatrix}
0 & -1 & \cdots & -1 \\
  & 1 \\
  &    & \ddots \\
  &    &        & 1
\end{bmatrix}_{n\times n}}} = \bigg[ \bold{0} \ \big| \ \bold{x}_2 - \bold{x}_1 \ \big| \ \ldots \ \big|  \ \bold{x}_n - \bold{x}_1 \bigg] \ . 
\end{equation}
On the other hand, time-warping extensions and contractions are described as follows:
\begin{equation}\label{eq:matrix time-warping extension}
\bigg[ \bold{x}_1 \big| \cdots \big| \bold{x}_n \bigg] \cdot {\scriptsize{\begin{bmatrix}
1 & \\
  & 1 & 1 \\
  &   &   & 1 \\
  &   &   &   & \ddots \\
  &   &   &   &        & 1
\end{bmatrix}_{n \times (n+1)}}} = \bigg[ \bold{x}_1 \big| \bold{x}_2 \big| \bold{x}_2 \big| \bold{x}_3 \big| \cdots \big| \bold{x}_n \bigg] 
\end{equation}
\begin{equation}\label{eq:matrix time-warping contraction}
\bigg[ \bold{x}_1 \big| \bold{x}_2 \big| \bold{x}_2 \big| \bold{x}_3 \big| \cdots \big| \bold{x}_n \bigg] \cdot {\scriptsize{\begin{bmatrix}
1 & \\
  & 1 & -1 \\
  &   &  1 \\
  &   &  1 \\
  &   &     &  \ddots  \\
  &   &     &          & 1
\end{bmatrix}_{(n+1)\times n}}} = \bigg[ \bold{x}_1 \big| \bold{x}_2 \big| \bold{x}_3 \big| \cdots \big| \bold{x}_n \bigg] \ . 
\end{equation}
Finally, the time-series of consecutive differences $\Delta\underline{\bold{x}}$ is obtained by: 
\begin{equation}\label{eq:matrix delta}
\bold{X}\cdot {\scriptsize{\begin{bmatrix}
-1 \\
1 & -1 \\
  & 1 & \ddots \\
  &   & \ddots & -1 \\
  &   &        & 1
\end{bmatrix}_{n\times (n-1)}}} = \bigg[  \bold{x}_2 - \bold{x}_1 \ \big| \ \ldots \ \big|  \ \bold{x}_n - \bold{x}_{n-1} \bigg] \ . 
\end{equation}
\end{remark}

\indent Similarly to concatenation of paths, one can consider a {\em shifted} concatenation of time-series: for any two time-series $\underline{\bold{x}}=\{\bold{x}_1,\ldots, \bold{x}_n\}$ and $\underline{\bold{y}}=\{\bold{y}_1,\ldots , \bold{y}_m\}$ in $\mathbb R^d$, their (shifted) {\em concatenation} is the time-series of length $n+m$

\begin{align}\label{def:concatenation time-series}
\{\underline{\bold{x}}|\underline{\bold{y}}\} & := \begin{cases} 
\bold{x}_k & k \in [n] \\
\bold{y}_{k-n}-\bold{y}_1 + \bold{x}_n & k \in [n+m]\setminus [n]
\end{cases} \\
& = \{ \bold{x}_1, \ldots, \bold{x}_n , \bold{x}_n , \bold{y}_2-\bold{y}_1+\bold{x}_n , \ldots, \bold{y}_m - \bold{y}_1 + \bold{x}_n \} \nonumber
\end{align}

\noindent The shift in the formula \eqref{def:concatenation time-series} is motivated by the need of emulating the Chen relation \eqref{eq:chen relation}. Indeed, for discrete signatures of time-series the following {\em discrete Chen relation} holds (cf. \cite[Theorem 3.6]{bellingeri2025discrete}):
\begin{equation}\label{eq:discrete Chen relation}
\Sigma(\{\underline{\bold{x}}|\underline{\bold{y}}\}) = \Sigma(\underline{\bold{x}})\otimes \Sigma(\underline{\bold{y}}) \ . 
\end{equation}

\noindent Note that the shifted concatenation of time-series is coherent with the time-warping-translation equivalence. Moreover, the class $\{\bold{0}\}$ (modulo $\sim_{twt}$) is the identity element in $\TS(\mathbb R^d)$ for the shifted concatenation:
\[ \{\{\bold{0}\}|\underline{\bold{x}}\} = \{ \underline{\bold{x}}|\{\bold{0}\} \} = \underline{\bold{x}} \ . \]
Finally, $\TS(\mathbb R^d)$ is a {\em monoid} with respect to the shift concatenation with identity element $\{\bold{0}\}$, and the discrete signature map \eqref{eq:discrete signature map} is a {\em monoid homomorphism} (in light of \eqref{eq:discrete Chen relation} and \eqref{eq:monoid homomorphism}). Now, determining when two time-series have the same discrete signature is equivalent to determine the kernel of the monoid homomorphism $\Sigma$
\[ \ker(\Sigma):=\left\{ (\underline{\bold{x}},\underline{\bold{y}}) \in \TS(\mathbb R^d)^2 \ | \ \Sigma(\underline{\bold{x}})=\Sigma(\underline{\bold{y}}) \right\} \ . \]
\noindent Note that this is different from determining which time-series have trivial discrete signature $\epsilon$. We claim that $\ker(\Sigma)$ corresponds to the diagonal in $(\TS(\mathbb R^d)/_{\sim_{twt}})^2$, and hence the (quotient) discrete signature map $\TS(\mathbb R^d)/_{\sim_{twt}} \rightarrow \T((\mathscr{V}_d))$ is injective. \\
\hfill\break
\indent In the case of the classical Chen signatures, the bijection between tree-like equivalence for paths and the Chen signatures is obtained by exploiting the group structure of paths and the Chen relation \eqref{eq:chen relation} making the signature path map \eqref{eq:signature path} a group homomorphism. For discrete signature we just have a monoid structure, so determining the (monoidal) kernel $\ker(\Sigma)$ requires a different approach. First, we observe that TWR time-series with same discrete signatures must have the same length.\\
\indent Given a time-series $\underline{\bold{x}}=\{\bold{x}_1,\ldots, \bold{x}_n\}\in \TS(\mathbb{R}^d)$ we call {\em difference time-series} the time-series of consecutive differences 
\[ \Delta\underline{\bold{x}}:=(\Delta \bold{x}_1, \ldots , \Delta \bold{x}_{n-1}) \ \ \  \text{where}  \ \Delta \bold{x}_j:= \bold{x}_{j+1}-\bold{x}_j \ . \]
If $\underline{\bold{x}}$ is TWR, then $\Delta\underline{\bold{x}}$ has no zero vectors. Observe that two TWR time-series $\underline{\bold{x}},\underline{\bold{y}} \in \TS(\mathbb R^d)$ with same discrete signature $\Sigma(\underline{\bold{x}})=\Sigma(\underline{\bold{y}})$ must have same length: indeed, it is enough to consider the coefficient $C^{\underline{\bold{x}}}_{Z_{i_1}\otimes Z_{i_2} \otimes \cdots \otimes Z_{i_n}} = \Delta x_{i_1 1} \cdot \Delta x_{i_2 2} \cdots \Delta x_{i_n n}$ for $(i_1,\ldots, i_n)\subset [d]^n$ such that $\Delta x_{i_1 1}, \Delta x_{i_2 2}, ..., \Delta x_{i_n n}$ are nonzero entries. \\

Diehl, Ebrahimi-Fard and Tapia \cite[Theorem 3.15]{diehl2022tropical} addressed the injectivity of the discrete signature up to $\sim_{twt}$ in a more general context, working over any semiring $\mathbb{S}$ with cancellation property. In their proof, they use an elegant inductive argument (based on the cancellation property of the semiring) to show the injectivity for $1$-dimensional time-series.

\begin{proposition}[\cite{diehl2022tropical}, Theorem 3.15]\label{lemma:1-dim time-warping injectivity}
For any two $1$-dimensional time-series ${\bold{x}},{\bold{y}} \in \TS(\mathbb R)$,
\[ \Sigma({\bold{x}}) = \Sigma({\bold{y}}) \ \iff \ {\bold{x}} \sim_{twt} {\bold{y}} \ . \]
In particular, the quotient map $\TS(\mathbb R)/_{\sim_{twt}} \hookrightarrow \T((\mathscr{V}_1))$ is injective.
\end{proposition}

However, deducing the injectivity for higher dimensional time-series from the one for $1$-dimensional time-series requires an extra necessary step which is not taken into account in the proof of \cite[Theorem 3.15]{diehl2022tropical}. Indeed, starting from two TWR time-series $\underline{\bold{x}}, \underline{\bold{y}} \in \TS(\mathbb R^d)$, for any $i\in [d]$ the $1$-dimensional time-series $\underline{\bold{x}}^i:=\{x_{i1}, \ldots, x_{in}\}$ and $\underline{\bold{y}}^i:=\{y_{i1}, \ldots , y_{in}\}$ in $\TS(\mathbb R)$ of the $i$--th vector entries might be non-TWR. Then the same argument to prove Proposition \ref{lemma:1-dim time-warping injectivity} shows that for any $i\in [d]$ the $1$-dimensional difference time-series $\Delta\underline{\bold{x}}^i$ and $\Delta\underline{\bold{y}}^i$ have the same (ordered) nonzero values, but a priori they can still have zero values appearing at different entries. We fill this gap and prove injectivity for higher dimensional time-series. The following proof works by replacing $\mathbb R$ with any semiring $\mathbb S$ with the cancellation property.

\begin{theorem}\label{thm:uniqueness up to time-warping}
For any two time-series $\underline{\bold{x}},\underline{\bold{y}} \in \TS(\mathbb R^d)$,
\[ \Sigma(\underline{\bold{x}}) = \Sigma(\underline{\bold{y}}) \ \iff \ \underline{\bold{x}} \sim_{twt} \underline{\bold{y}} \ . \]
In particular, the quotient map $\TS(\mathbb R^d)/_{\sim_{twt}} \hookrightarrow \T((\mathscr{V}_d))$ is injective.
\end{theorem}

\begin{proof}
We assume $\underline{\bold{x}}, \underline{\bold{y}} \in \TS(\mathbb R^d)$ to be TWR and such that $\Sigma(\underline{\bold{x}})=\Sigma(\underline{\bold{y}})$. We know that $\underline{\bold{x}}$ and $\underline{\bold{y}}$ must have same length. Write $\underline{\bold{x}}=\{\bold{x}_1,\ldots, \bold{x}_n\}$ and $\underline{\bold{y}}=\{\bold{y}_1,\ldots , \bold{y}_n\}$.\\
\indent For every $i \in [d]$ consider the $1$-dimensional time-series $\underline{\bold{x}}^i:=\{x_{i1}, \ldots, x_{in}\}$ and $\underline{\bold{y}}^i:=\{y_{i1}, \ldots , y_{in}\}$ of the $i$--th vector entries. They can be non-TWR. However, the same argument in the proof of \cite[Theorem 3.15]{diehl2022tropical} shows that for any $i\in [d]$ the $1$-dimensional difference time-series $\Delta\underline{\bold{x}}^i$ and $\Delta\underline{\bold{y}}^i$ are equivalent, that is they have the same (ordered) nonzero values, and they can only differ in the order of appearance of the zeroes. In the following we prove that they also have zero values at the same entries, implying that $\Delta\underline{\bold{x}}^i=\Delta \underline{\bold{y}}^i$.

\indent Since the $d$-dimensional time-series $\underline{\bold{x}}$ is TWR, the difference time-series $\Delta\underline{\bold{x}}$ doesn't have zero vectors, hence there exists $(i_1,\ldots, i_n)\subset [d]^n$ such that $C^{\underline{\bold{x}}}_{Z_{i_1}\otimes \ldots \otimes Z_{i_n}}= \Delta x_{i_1 1}\cdots \Delta x_{i_n n}\neq 0$. In particular, each $\Delta x_{i_j j}\neq 0$. Since the time-series have the same discrete signatures, it follows $C_{Z_{i_1}\otimes \ldots \otimes Z_{i_n}}^{\underline{\bold{y}}} = \Delta y_{i_1 1}\cdots \Delta y_{i_n n}$ is nonzero too, as well as each $\Delta y_{i_j j}$. Now, pick $(\bar{i}, \bar{j}) \in [d]\times n$ such that $\Delta x_{\bar{i} \bar{j}}=0$ and substitute in the above tensor-monomials $Z_{i_{\bar{j}} \bar{j}}$ with $Z_{\bar{i} \bar{j}}$: one gets
\begin{align*}
    0 & = \Delta x_{i_1 1} \cdots \Delta x_{i_{\bar{j}-1} \, \bar{j}-1} \cdot \Delta x_{\bar{i} \bar{j}} \cdot \Delta x_{i_{\bar{j}+1} \, \bar{j}} \cdots  \Delta x_{i_n n}\\
    & = \underbrace{\Delta y_{i_1 1} \cdots \Delta y_{i_{\bar{j}-1} \, \bar{j}-1}}_{\neq 0} \cdot \Delta y_{\bar{i} \bar{j}} \cdot \underbrace{\Delta y_{i_{\bar{j}+1} \, \bar{j}} \cdots  \Delta y_{i_n n}}_{\neq 0}
    \end{align*}
implying $\Delta y_{\bar{i} \bar{j}}=0$. This shows that the $1$-dimensional difference time-series $\Delta \underline{\bold{x}}^i$ and $\Delta \underline{\bold{y}}^i$ have zeroes at the same entries, hence they coincide. We conclude that $\Delta \underline{\bold{x}}=\Delta \underline{\bold{y}}$, that is $\underline{\bold{x}}\sim_{twt} \underline{\bold{y}}$.
\end{proof}

\begin{remark}[Refinement of segment decomposition]\label{rmk:refinement} 
For data applications, unlike the classical signature of paths, the discrete signature is not invariant under \textquotedblleft refinement of segment decomposition\textquotedblright, which would correspond to an expansion of the time-series: an evidence is given by the discrete Chen relation \eqref{eq:discrete Chen relation}. \textcolor{black}{For instance, a segment $[b,d]$ and its refinement $[b,\frac{b+d}{2}]\cup[\frac{b+d}{2},d]$ give respectively the time-series $\{b,d\}$ and $\big\{ \{b,\frac{b+d}{2}\} | \{\frac{b+d}{2}, d\} \big\} = \{b, \frac{b+d}{2}, \frac{b+d}{2} , d\}$, whose discrete signatures are such that
\begin{align*} 
    &C^{\{b,d\}}_{Z^2}  = (d-b)^2,  \\
    &C^{\{ \{b,\frac{b+d}{2}\} | \{\frac{b+d}{2}, d\} \}}_{Z^2}  = C^{\{b,\frac{b+d}{2}\}}_\epsilon \cdot C^{\{\frac{b+d}{2}, d\}}_{Z^2} \ + \ C^{\{b,\frac{b+d}{2}\}}_{Z^2} \cdot C^{\{\frac{b+d}{2}, d\}}_\epsilon  = \ \textstyle{\frac{(d-b)^2}{2}}.
    \end{align*}  
    }
\noindent However, the Chen and the discrete signatures are interrelated \cite[Proposition 5.5]{diehl2020timewarping}: given a bounded variation path $\gamma:[0,1]\rightarrow \mathbb R^d$, the limit ---over finite partitions $\pi=\{ 0 < t_1 < \ldots < t_n <1\}$ --- of the discrete signatures of the time-series $\underline{\bold{x}}^\pi:=\{ \gamma(0), \gamma (t_1),\ldots, \gamma(t_n),\gamma(1)\}$ converges to the Chen signature of $\gamma$.
\end{remark}

\subsection{The discrete landscape feature map}\label{subsec:dis_sig_land}

The discrete signature is a well-defined and meaningful descriptor for persistence landscapes (and in general, of PWL paths). Indeed, persistence landscapes come with a \textquotedblleft canonical\textquotedblright \space time-series associated to them, namely the {\em minimal segment decomposition} (cf.\ Definition \ref{def:minimal segment decomp}). \\
\indent Each level $\lambda_k(t)$ of a persistence landscape $\lambda$ is a piecewise linear function with slope $1,-1,$ or $0$. The \textit{critical points} of $\lambda_k$ are the values of $t$ at which the slope changes, and the \textit{critical points} of the landscape $\lambda$ is the union of all the critical points for each $\lambda_k$. A persistence landscape is completely determined by its critical points. In fact, the standard implementation of landscape methods keeps track of the critical points rather than the landscape as a function \cite{bubenik2017persistence}. One can list the critical points of a landscape in terms of the corresponding barcode as follows.

\begin{lemma}[\cite{Bubenik_2020_landscapes}, Lemma 5.8]
    Given a barcode $B=\{ (b_i,d_i) \}$, the critical points in the corresponding landscape $\lambda^B$ consist of
    \begin{enumerate}
        \item the left endpoints $b_i$,
        \item the right endpoints $d_i$,
        \item the midpoints $\frac{b_i+d_i}{2}$,
        \item the midpoints $\frac{b_i + b_j}{2}$ of pairs of bars where $b_j < b_i <d_j<d_i$.
    \end{enumerate}
\end{lemma}

\begin{definition}\label{def:critical points}
Given a landscape $\lambda = (\lambda_1,...\lambda_d)$, we denote by $\underline{\boldsymbol{\lambda}} \in \TS(\mathbb R^d)$ the time-series obtained by evaluating $\lambda$ at all of its critical points, i.e. $\underline{\boldsymbol{\lambda}}$ is the sequence of critical values of $\lambda$. An example is given in Example \ref{example: signatures weight 2}.
\end{definition}

\begin{remark}\label{rmk:fix number of levels}
In practice, we work with landscapes with a fixed number of levels, say $d$, where we truncate levels, or extend levels that are identically zero as necessary to reach $d$ levels. 
\end{remark}

The assumption in the above remark allows us to associate a time-series to a landscape (hence to a barcode) in a canonical way. We call this \textit{discrete landscape embedding} and denote it by
\begin{equation}\label{eq:discrete embedding}
I: \BBar \rightarrow \TS(\mathbb R^d)
\end{equation}
Analogous to the feature maps in Subsection \ref{subsec:preliminaries landscapes}, we define the \textit{discrete landscape feature map} as the composition of the discrete signature and the discrete landscape embedding:
\begin{equation}\label{eq:discrete feature map}
\Phi_I: \BBar \xrightarrow{I} \TS(\mathbb R^d) \xrightarrow{\Sigma} \T((\mathscr{V}_d)).
\end{equation}

The following naive example illustrates how the discrete feature map distinguish between landscapes that are tree-like equivalent.

\begin{example}
    Consider the singleton barcode $B = \{ (b,d) \}$. Then $I(B)=\left(0,\frac{ d-b}{{2}},0\right)$ is a one-dimensional time-series of length $3$ whose discrete signature up to weight $3$ is
    \begin{align*}
        & & \langle \Phi_I(B), Z \rangle &= 0 \\
        \langle \Phi_I(B), Z^2 \rangle &= \frac{(d-b)^2}{2} 
        & \langle \Phi_I(B), Z \otimes Z \rangle  &= - \frac{(d-b)^2}{2^2} & \langle\Phi_I(B), Z^3\rangle &= 0 \\
        \langle\Phi_I(B), Z^2 \otimes Z\rangle &= - \frac{(d-b)^3}{2^{3}}
        & \langle\Phi_I(B), Z\otimes Z^2\rangle &= \frac{(d-b)^3}{2^{3}}& \langle\Phi_I(B), Z\otimes Z \otimes Z\rangle &= 0
    \end{align*}
    Now consider a barcode $C = \{ (b,d),(b',d') \}$ where $b' > d$ (similar to Figure \ref{fig:two_spheres}). Then, up to $twt$--equivalence, 
    \[I(C) = \left( 0, \frac{d-b}{{2}},0, \frac{d'-b'}{{2}} ,0 \right) \] 
    is a one-dimensional time-series of length $5$ whose signature up to weight 3 is
        \begin{align*}
        \langle\Phi_I(C), Z\rangle &= 0
        &
        \langle\Phi_I(C), Z \otimes Z\rangle &= -\frac{1}{2^2}\left( (d-b)^2 + (d'-b')^2 \right) \\
        \langle\Phi_I(C), Z^2\rangle &= \frac{1}{2}\left((d-b)^2 + (d'-b')^2\right) 
        &
        \langle\Phi_I(C), Z^2 \otimes Z\rangle &= -\frac{1}{2^3}\left( (d-b)^3 + (d'-b')^3 \right) \\
        \langle\Phi_I(C), Z^3\rangle &= 0
        &
        \langle\Phi_I(C), Z\otimes Z^2\rangle &=  \frac{1}{2^3}\left( (d-b)^3 + (d'-b')^3 \right) \\
        & & \langle\Phi_I(C), Z\otimes Z \otimes Z\rangle &= 0 
    \end{align*}
\end{example}

In slightly more generality, whenever a barcode $B$ has a landscape $\lambda^B$ with a single level, the discrete signature captures persistence, intended as the difference values between death and birth times. 

\begin{lemma}\label{lemma:1_level_sigs}
    The signature of a barcode $B$ defining a single-level landscape $\lambda^B$ is completely determined by the persistence of the bars in $B$.
\end{lemma}
\begin{proof}
    If $\lambda^B$ has a single level then there exists a total ordering $\prec$ of $B$ so that $d \leq b'$ whenever $(d,b) \prec (d',b')$, since otherwise there would be a second level. Label the barcodes in such a way that $(b_1,d_1)<(b_2,d_2) < ... < (b_{|B|},d_{|B|}))$, with $d_i<b_j$ if $i<j$. Then the time-series of differences has the form 
    {\small \[
    \Delta I(B) = \left(\frac{(d_1-b_1)}{{2}}, \frac{-(d_1-b_1)}{{2}},\frac{(d_2-b_2)}{{2}},\frac{-(d_2-b_2)}{{2}},..., \frac{(d_{|B|}-b_{|B|})}{{2}},\frac{-(d_{|B|}-b_{|B|})}{{2}} \right).
    \] }
    Since the signature is obtained by evaluation of polynomials on the time-series of differences, the lemma follows.
\end{proof}

\begin{corollary}
    The coefficients corresponding to univariate polynomials $p \in \mathbb{R}[Z_i]$ are determined by the persistence.
\end{corollary}

\indent One can derive some explicit signature coefficients in terms of persistence for single-level landscapes. For example, straightforward computations with the difference time-series in the proof above give
    
    \begin{align*} 
    \langle\Phi_I(B), Z^{\otimes n} \rangle &= 0 \hspace{26mm} \text{if } n \equiv 1 \ (mod \ 2)  \\
    \langle\Phi_I(B), Z^n \rangle &= \begin{cases}
        0 & \text{if } n \equiv 1 \ (mod \ 2) \\
        \sum\limits_{\beta \in B} \frac{(d_\beta-b_\beta)^{n}}{2^{n-1}} & \text{if } n \equiv 0 \ (mod \ 2)
        \end{cases}
        \end{align*}
        
\begin{remark}
The latter equation can be interpreted as a list of even moments of the time-series, so that standard statistical descriptors of time-series are naturally captured by the discrete signature. In particular, going back to the example in Figure \ref{fig:two_spheres}, the above observation shows that the discrete signature encodes information about the persistence of the spheres, which is determined by their radii, whereas the Chen signature is trivial.
\end{remark}

\indent The dependence of the signature on persistence alone is specific to one-level landscapes. For landscapes with more levels the signature encodes additional information related to relative distances between bars, and this is picked up already at weight 2.

\begin{example}\label{example: signatures weight 2}
    Consider the time-series
    \[ \underline{\bold{x}} = \left(
    \begin{pmatrix}
0 \\ 0
\end{pmatrix},
\begin{pmatrix}
1 \\ 0
\end{pmatrix},
\begin{pmatrix}
1/2 \\ 1/2
\end{pmatrix},
\begin{pmatrix}
1 \\ 0
\end{pmatrix},
\begin{pmatrix}
3 \\ 0
\end{pmatrix},
\begin{pmatrix}
0 \\ 0
\end{pmatrix}
    \right) \]
    corresponding to the landscape
    
\begin{center}
\begin{tikzpicture}[scale=0.6]

    \draw[thick, blue] (0,0) -- (1,1) -- (1.5,0.5)--(4,3) -- (7,0);

    \draw[thick, red] (0,0)--(1,0)-- (1.5,0.5) -- (2,0) -- (7,0);

    \fill[blue] (1,1) circle[radius=2pt] node[above left] {$(1,1)$};
    \fill[blue] (4,3) circle[radius=2pt] node[above right] {$(4,3)$};
\end{tikzpicture}
\end{center}
    
    \noindent The discrete signature of $\underline{\bold{x}}$ up to weight $2$ is
    \[ \langle\Sigma(\underline{\bold{x}}), Z_1\rangle = \langle\Sigma(\underline{\bold{x}}), Z_2\rangle  = 0 \]
    \begin{align*} 
    \langle\Sigma(\underline{\bold{x}}), Z_1^2\rangle &= \frac{29}{2}
    & \langle\Sigma(\underline{\bold{x}}), Z_1 Z_2\rangle &= - \frac{1}{2}
    & \langle\Sigma(\underline{\bold{x}}), Z_2^2\rangle &= \frac{1}{2} \\
    \langle\Sigma(\underline{\bold{x}}), Z_1\otimes Z_1\rangle &=-\frac{29}{4} 
    & \langle\Sigma(\underline{\bold{x}}), Z_1 \otimes Z_2\rangle &= \frac{1}{4}
    & \langle\Sigma(\underline{\bold{x}}), Z_2\otimes Z_2\rangle &= -\frac{1}{4}
    \end{align*}
\end{example}

\subsection{Continuity of the discrete landscape feature map}
The discrete signature, being a component-wise polynomial map, is continuous. However, the discrete landscape embedding is not, hence the discrete landscape feature map is not continuous either. The key obstruction is illustrated in the following example.

\begin{example}
    Let $B =\{ (b,d)\}$, with $d>b$, and $B' = \{(b,d),(\frac{d-b}{4}, \frac{d-b}{4} + \delta)\}$ for some $0< \delta << d-b$. Their bottleneck distance (cf.\ Section \ref{subsec:restricted}) is $\dist_{\B}(B,B')= \frac{\delta}{\sqrt{2}}$, and the corresponding persistence landscapes are
    \begin{figure}[H]
    \centering
    \begin{tikzpicture}[scale=0.3]
    \draw[thick, blue] (0,0) -- (4,4) -- (8,0);
    \draw[thick, blue] (11,0) -- (15,4) -- (19,0);
    \draw[thick, red] (11,0) -- (13,0) -- (13.3,0.3) -- (13.6,0) -- (19,0);
    \node at (4,-1) {$B$};
    \node at (15,-1) {$B'$};
    \end{tikzpicture}
    \end{figure}
    \noindent Some coefficients in the discrete signatures are such that no choice of $\delta$ will make their difference small. For instance, 
    \begin{align*}
    \langle \Phi_I(B) ,  Z_1 \otimes Z_1 \otimes Z_1 \rangle & = 0 \\
    \langle \Phi_I(B') , Z_1 \otimes Z_1 \otimes Z_1  \rangle & =  -\frac{(d-b)^3}{2^5} + o(\delta) \ . 
    \end{align*}
\end{example}

The heart of the issue with continuity is the poor behaviour of the discrete signature under refinement of segment decomposition (cf.\ Remark \ref{rmk:refinement}). Indeed, in the above example, the inclusion of a bar with arbitrarily small persistence, which produces a barcode that is close in bottleneck distance, forces a refinement of segment decomposition, whence the discontinuity.

\subsection{The specialised discrete landscape feature map}\label{subsec:restricted}

For applications, one may want to compare barcodes using signatures as input for machine learning methods \cite{lyons2025signaturemethodsmachinelearning}. In light of this, it makes sense to consider a finite subset $D \subset \BBar$. Then we can modify the discrete embedding map $I: \BBar \rightarrow \TS(\mathbb R^d)$ \eqref{eq:discrete embedding} as follows. Consider the union of all critical points of the landscapes arising from the domain $D$
\[ C_D := \{ c_1,c_2,...,c_N \} =\{ c \ | \ c \ \text{critical point of} \ \lambda^B \text{ for some } B \in D \} \ , \]
and define the {\em specialised discrete embedding map} 
\begin{equation}\label{eq:restricted discrete embedding}
\begin{matrix}
I_D: & D & \rightarrow & \BV([0,L],\mathbb R^d) & \rightarrow & \TS(\mathbb R^d) \\
& B & \mapsto & \lambda^B & \mapsto & (\lambda^B (c_1), \lambda^B(c_2),..., \lambda^B(c_N))
\end{matrix} \ .
\end{equation}
This restriction allows us to have a consistent segment subdivision of the landscapes across all barcodes in $D$, which then guarantees continuity of the {\em specialised discrete landscape feature map}
\begin{equation}\label{eq:restricted discrteet feature map}
\Phi_{I_D}^{\leq k}: D \xrightarrow{I_D} \TS(\mathbb R^d) \xrightarrow{\Sigma^{\leq k}} \T^{\leq k}((\mathscr{V}_d))
\end{equation}
where $\T^{\leq k}((\mathscr{V}_d))$ denotes the space of discrete signature tensors truncated at weight $k$. \\
\indent For practical purposes, a particularly desirable property of a feature map would be stability, or equivalently being $1$-Lipschitz. A weaker but still desirable condition is uniform continuity. We aim that the specialised discrete landscape feature map $\Phi^{\leq k}_{I_D}$ is uniformly continuous. Before proving this, we describes explicitly the metrics on all the involved spaces, that is $\BBar$, $\BV([0,L],\mathbb R^d)$, $\TS(\mathbb R^d)$ and $\T^{\leq k}((\mathscr{V}_d))$.

\paragraph{Bottleneck distance.} The {\em bottleneck distance} was classically introduced for persistence diagrams in \cite[Sec.\ 3]{cohen2007stability}, and it translates in the following distance between barcodes:

\[ \dist_{\B}(B,B') := \inf_{\tau} \sup_{B_i \in B} \norm{ B_i - \tau(B_i)}_\infty \]

\noindent where $\tau$ runs among the bijections $\tau: B \cup \{(b,b)|b \in \mathbb R\} \rightarrow B' \cup \{(b,b)| b \in \mathbb R\}$ allowing matching birth-death points to diagonal points. This makes the definition consistent when the two barcodes have different number of intervals. The distance between landscapes is the $\infty$--norm $\norm{\cdot}_{\infty}$, as considered in \cite[Sec.\ 5]{bubenik2015statistical}:

\[ \dist_{\infty}(\lambda^B , \lambda^{B'}) := \norm{\lambda^B - \lambda^{B'}}_\infty = \sup_{t} \left\{ \norm{\lambda^B(t)-\lambda^{B'}(t)}_2 \right\} \ . \]

\noindent Finally, given two time-series in $\mathbb R^d$ of the same length $n$, the distance between them that we consider is

\[ \dist_{\TS}(\underline{\bold{x}}, \underline{\bold{y}}) = \max \left\{ \norm{\bold{x}_i-\bold{y}_i}_2 \ | \ i \in [n] \right\}  \ . \]

\paragraph{Bombieri-Weyl product.} The (non-degenerate) scalar product $\langle \cdot , \cdot \rangle$ on $\mathbb R^d$ defining the $2$--norm $\norm{\cdot}_2$ naturally induces a scalar product on $\T^{\leq k}((\mathscr{V}_d))$, called {\em Bombieri-Weyl product}. Fix an orthonormal basis $(e_1,\ldots , e_d)$ of $\mathbb R^d$. Let $Sym^\bullet \mathbb R^d$ be the graded symmetric algebra over $\mathbb R^d$, which is isomorphic (as graded algebras) to the polynomial ring $\mathbb R[Z_1,\ldots, Z_d]$: one identifies the basis vector $e_i$ with the variable $Z_i$. In particular, the isomorphism restricts to each graded component, so that $Sym^h\mathbb R^d$ is isomorphic to the vector space $\mathbb R[Z_1,\ldots, Z_d]_h$ of homogeneous polynomials in $d$ variables of degree $h$. For any $h\geq 0$, the scalar product $\langle \cdot ,\cdot \rangle$ extends to a scalar product, called {\em Bombieri-Weyl product} (BW-product), on $\mathbb R[Z_1,\ldots, Z_d]_h$ defined on the monomials (and extended by linearity) by
\[ \langle Z_1^{\alpha_1}\cdots Z_d^{\alpha_d} , Z_1^{\beta_1}\cdots Z_d^{\beta_d}\rangle_{sym} = \binom{h}{\boldsymbol{\alpha}}^{-1}\delta_{\boldsymbol{\alpha}\boldsymbol{\beta}} \ , \]
where $\binom{h}{\boldsymbol{\alpha}}=\frac{h!}{\alpha_1! \cdots \alpha_d!}$ and the latter symbol is the Kronecker symbol. The BW-product on $Sym^h\mathbb R^d$ induces the Bombieri-Weyl norm, also known as {\em Kostlan norm}: for any degree--$h$ monomial $\bold{Z}^{\boldsymbol{\alpha}}:=Z_1^{\alpha_1}\cdots Z_d^{\alpha_d}$ it holds
\[ \norm{\bold{Z}^{\boldsymbol{\alpha}}}_{sym}^2 = \binom{h}{\boldsymbol{\alpha}}^{-1} \ .\]
One extends the BW-product to the whole algebra $Sym^\bullet\mathbb R^d\simeq \mathbb R[Z_1, \ldots , Z_d]$ by imposing orthogonality between monomials with different grades. This defines a scalar product on $\mathscr{V}_d$. Again, this extends to any tensor-power of $\mathscr{V}_d^{\otimes \ell}$, resulting in the following norm: given $\ell$ monomials $\bold{Z}^{\boldsymbol{\alpha}_1}, \ldots, \bold{Z}^{\boldsymbol{\alpha}_\ell}$ in $\mathscr{V}_d$ such that each $\boldsymbol{\alpha}_i=(\alpha_{i,1}, \ldots, \alpha_{i,d})\in \mathbb N^d$ is a partition of the degree $|\boldsymbol{\alpha}_i|=\alpha_{i,1}+\ldots + \alpha_{i,d}$, it holds
\[ \norm{\bold{Z}^{\boldsymbol{\alpha}_1}\otimes \ldots \otimes \bold{Z}^{\boldsymbol{\alpha}_\ell}}_{BW}^2 = \norm{\bold{Z}^{\boldsymbol{\alpha}_1}}_{sym}^2\cdots \norm{\bold{Z}^{\boldsymbol{\alpha}_\ell}}_{sym}^2 = \binom{|\boldsymbol{\alpha}_1|}{\boldsymbol{\alpha}_1}^{-1}\cdots \binom{|\boldsymbol{\alpha}_\ell|}{\boldsymbol{\alpha}_{\ell}}^{-1} \ . \]
Analogously, one extends the scalar product to $\T^{\leq k}((\mathscr{V}_d))$ by imposing orthogonality between tensor-products of different order. In particular, the BW-norm on $\T^{\leq k}((\mathscr{V}_d))$ is

{\footnotesize \[ \norm{\sum_{|\boldsymbol{\alpha}_1|+\ldots + |\boldsymbol{\alpha}_\ell|\leq k} C_{\boldsymbol{\alpha}_1, \ldots, \boldsymbol{\alpha}_\ell}\bold{Z}^{\boldsymbol{\alpha}_1}\otimes \ldots \otimes \bold{Z}^{\boldsymbol{\alpha}_\ell} }_{BW}^2 = \sum_{|\boldsymbol{\alpha}_1|+\ldots + |\boldsymbol{\alpha}_\ell|\leq k} C_{\boldsymbol{\alpha}_1, \ldots, \boldsymbol{\alpha}_\ell}^2 \binom{|\boldsymbol{\alpha}_1|}{\boldsymbol{\alpha}_1}^{-1}\cdots \binom{|\boldsymbol{\alpha}_\ell|}{\boldsymbol{\alpha}_{\ell}}^{-1} \ . \]}

This allows to define the following {\em Bombieri-Weyl distance} on $\T^{\leq k}((\mathscr{V}_d))$: given two elements $f,g\in \T^{\leq k}((\mathscr{V}_d))$ with coefficients $C^f_{\boldsymbol{\alpha}_1, \ldots, \boldsymbol{\alpha}_\ell}$ and $C^g_{\boldsymbol{\alpha}_1, \ldots, \boldsymbol{\alpha}_\ell}$ respectively, it holds
\begin{align}\label{eq:BW norm}
    \dist_{BW}(f,g) & = \norm{\sum_{|\boldsymbol{\alpha}_1|+\ldots + |\boldsymbol{\alpha}_\ell|\leq k} \left( C^f_{\boldsymbol{\alpha}_1, \ldots, \boldsymbol{\alpha}_\ell} - C^g_{\boldsymbol{\alpha}_1, \ldots, \boldsymbol{\alpha}_\ell} \right)\bold{Z}^{\boldsymbol{\alpha}_1}\otimes \ldots \otimes \bold{Z}^{\boldsymbol{\alpha}_\ell} }_{BW}^2 \nonumber \\
    & = \sum_{|\boldsymbol{\alpha}_1|+\ldots + |\boldsymbol{\alpha}_\ell|\leq k} \left( C^f_{\boldsymbol{\alpha}_1, \ldots, \boldsymbol{\alpha}_\ell} - C^g_{\boldsymbol{\alpha}_1, \ldots, \boldsymbol{\alpha}_\ell} \right)^2 \binom{|\boldsymbol{\alpha}_1|}{\boldsymbol{\alpha}_1}^{-1}\cdots \binom{|\boldsymbol{\alpha}_\ell|}{\boldsymbol{\alpha}_{\ell}}^{-1}
    \end{align}

\paragraph{Uniform continuity.} In our case we start from the standard scalar product $\langle \bold{u},\bold{w}\rangle = \sum u_iw_i$ on $\mathbb{R}^d$, defining the $2$--norm $\norm{\cdot}_2$.

\begin{theorem}\label{thm:uniform continuity}
    Given $D\subset \BBar$ a finite set of barcodes, the specialised discrete landscape feature map $\Phi_{I_D}^{\leq k}$ \eqref{eq:restricted discrteet feature map} is uniformly continuous with respect to the bottleneck distance on $\BBar$ and the Bombieri-Weyl distance on $\T^{\leq k}((\mathscr{V}_d))$.
\end{theorem}
\begin{proof}
We divide the proof in more steps, showing that each map involved in $\Phi_{I_D}^{\leq k}$ is uniformly continous.\\
\indent {\em Claim 1: The discrete embedding $I_D$ is uniformly continuous.} The map $I_D$ can be factored as $I_d = \mathbf{ev} \circ \lambda$ where $\lambda: D \rightarrow \BV(\mathbb R^d)$ is the persistence landscape map restricted to $D$, and $\mathbf{ev}$ is evaluation at all the points in $C_D$. \\
In \cite[Theorem 2.4]{Bubenik_2020_landscapes}, it was shown that $\lambda$ is stable, meaning $1$-Lipschitz where $\BV(\mathbb R^d)$ has the $\infty$--norm: in particular, $\lambda$ is uniformly continuous. On the other hand, the evaluation map $\mathbf{ev}$ is evaluation of persistence landscapes (which are piecewise linear functions) at a fixed set of points, so it is also $1$-Lipschitz, hence uniformly continuous.\\
\indent {\em Claim 2: The discrete signature $\Sigma$ is uniformly continuous.} We need to prove that for any $\epsilon>0$ there exists $\gamma>0$ such that 
\[ \dist_{\TS}(\underline{\bold{x}},\underline{\bold{y}})<\gamma \ \implies \norm{\Sigma^{\leq k}(\underline{\bold{x}}) - \Sigma^{\leq k}(\underline{\bold{y}})}_{BW}< \epsilon \ . \]
From \eqref{eq:BW norm}, and since we are working with a truncation of the discrete signature, it is enough to show the statement for any coefficient $| C^{\underline{\bold{x}}}_{\boldsymbol{\alpha}_1, \ldots, \boldsymbol{\alpha}_\ell} - C^{\underline{\bold{y}}}_{\boldsymbol{\alpha}_1, \ldots, \boldsymbol{\alpha}_\ell} |$. \\
Given $\underline{\bold{x}}, \underline{\bold{y}}\in \TS(\mathbb R^d)$ of length $n$, consider their difference time-series $\Delta \underline{\bold{x}}, \Delta \underline{\bold{y}}$ of length $n-1$. Flatten them in the vectors $\boldsymbol{\Delta_x}:=(\Delta \bold{x}_1,\ldots, \Delta \bold{x}_{n-1})$ and $\boldsymbol{\Delta_y}:=(\Delta \bold{y}_1,\ldots, \Delta \bold{y}_{n-1})$ of length $d(n-1)$. In the polynomial ring $\mathbb R[T_{11} , \ldots ,T_{1d}, T_{21}, \ldots , T_{n-1 \, d}]$ consider the polynomial 

\[ P(\bold{T}) := \sum_{1 \leq i_1 \lneq \ldots \lneq i_{\ell} < n}  p_1(T_{i_1 \, 1},  \ldots , T_{i_1 \, d})\cdots p_\ell(T_{i_\ell \, 1}, \ldots, T_{i_\ell \, d}) \ . \]

\noindent Then it holds 

\[ P(\boldsymbol{\Delta_x}) = \sum_{1 \leq i_1 \lneq \ldots \lneq i_{\ell} < n} p_1(\Delta \bold{x}_{i_1})\cdots p_\ell(\Delta \bold{x}_{i_\ell}) = C^{\underline{\bold{x}}}_{p_1\otimes \ldots \otimes p_\ell}  \ . \]

\noindent But polynomials (on a compact domain) are uniformly continuous, hence for any $\epsilon >0$ there exists $\delta >0$ such that 
\[ \norm{\boldsymbol{\Delta_x} - \boldsymbol{\Delta_y}}_2 < \delta \ \implies \ \big| C^{\underline{\bold{x}}}_{p_1\otimes \ldots \otimes p_\ell} - C^{\underline{\bold{y}}}_{p_1\otimes \ldots \otimes p_{\ell}} \big|= \big| P(\boldsymbol{\Delta_x})-P(\boldsymbol{\Delta_y}) \big| < \epsilon \ . \]

\noindent Now, if 
\[ \dist_{\TS}(\underline{\bold{x}}, \underline{\bold{y}}) := \max \left\{ \norm{\bold{x}_i-\bold{y}_i}_2 \ | \ i \in [n] \right\} \leq \frac{\delta}{2(n-1)} \ , \]
\noindent one gets
\[ \norm{\Delta \bold{x}_{i+1} - \Delta \bold{y}_{i+1}}_2 =\norm{\bold{x}_{i+1}-\bold{x}_{i}-\bold{y}_{i+1}+\bold{y}_i}_2 \leq \norm{\bold{x}_{i+1}-\bold{y}_{i+1}}_2 + \norm{\bold{x}_i - \bold{y}_i}_2 \leq \frac{\delta}{n-1} \]
implying that on the flattenings it holds

\[ \norm{\boldsymbol{\Delta_x} - \boldsymbol{\Delta_y}}_2 \leq \sum_{j \in [n-1]} \norm{\Delta \bold{x}_{j} - \Delta \bold{y}_j}_2 \leq \delta \ . \]
We conclude that it is enough to take $\gamma := \frac{\delta}{2(n-1)}$ so that the hypothesis $\dist_{\TS}(\underline{\bold{x}},\underline{\bold{y}})<\gamma$ implies $|C^{\underline{\bold{x}}}_{\boldsymbol{\alpha}_1, \ldots, \boldsymbol{\alpha}_\ell} - C^{\underline{\bold{y}}}_{\boldsymbol{\alpha}_1, \ldots, \boldsymbol{\alpha}_\ell}| < \epsilon$. This proves that the truncated discrete signature map is uniformly continuous.
\end{proof}

We remark that discrete signatures characterise landscapes, hence barcodes, up to translating nonzero sections of a landscape along regions of the domain where the landscape is identically zero. Indeed, identically zero regions of the landscape produce consecutive zero vectors in the time-series of the critical points, and these can be contracted to a unique zero vector by time-warping. 
\section{Application to knotted proteins}\label{sec:application}

\noindent In this section we apply discrete signatures to a dataset of knotted proteins drawn from KnotProt 2.0 \cite{10.1093/nar/gky1140}, a curated database of proteins with entangled backbone structures. Knotted proteins have attracted considerable interest across structural biology and computational biology because the presence and localisation of entanglement are closely linked to protein folding pathways, structural stability, and functional behaviour \cite{faisca2015knotted,beccara}. In particular, even small topological changes in the backbone can lead to significant differences in stability and dynamics. Understanding and quantifying such entanglement therefore requires tools that capture both geometric and topological features of protein structure.

\indent We focus on proteins whose backbones form open-ended positive trefoil knots, the most common and well-studied class of knotted proteins. In this setting, the knot type is assigned using standard closure procedures for open curves, as implemented in KnotProt and related software. This provides a consistent topological framework in which to compare protein structures while retaining their geometric information.

\indent The data are obtained from the Protein Data Bank \cite{berman_protein_2000} and curated in KnotProt 2.0, which 
consists of the three-dimensional coordinates of the carbon (C-alpha) atoms of each protein backbone, and additional information about the backbone structure (i.e.~the ordering of C-alpha atoms along the chain), the full amino-acid sequence, and annotations such as knot type, knot core, and tail regions. 

\indent Following \cite{benjamin2023homology}, we interpolate five equidistant points between successive C-alpha atoms in order to obtain a sufficiently dense approximation of the piecewise linear backbone curve. We then compute one-dimensional persistent homology and persistence landscapes truncated at level $15$. The proteins are categorized into structural classes according to amino-acid {\em sequence similarity}; proteins with sufficiently similar sequences are expected to have similar three-dimensional structures. We adopt the same convention as \cite{benjamin2023homology} and label the class by a representative protein of the class, or \textquotedblleft other\textquotedblright \space whenever there were insufficient similar sequences to define additional classes. In the present analysis, we exclude proteins labelled \textquotedblleft other\textquotedblright \space from the sequence-similarity comparison, leaving $363$ proteins and $9$ structural classes.

\indent The knot depth is a standard geometric quantity used to quantify how far the knot core is located from the endpoints of the chain \cite{barbensi_f_2021}. More precisely, if $N$ and $C$ denote the N-tail and C-tail, respectively, and $T$ denotes the whole chain, then the depth of a protein chain $p$ is
\begin{equation}\label{eq:knot_depth}
D(p)=\frac{l(N)\,l(C)}{l(T)^2},
\end{equation}
where $l(\cdot)$ counts the number of C-alpha atoms in the relevant subchain. 

\indent In \cite{benjamin2023homology} the authors compute pairwise Wasserstein distances on persistence diagrams and pairwise $L_1$ distances on persistence landscapes, and apply Isomap to obtain low-dimensional embeddings of the protein dataset. These embeddings reveal a visible organisation of proteins according to sequence homology class and knot depth, providing qualitative evidence of a correlation between persistent homology features and biologically meaningful structure. 

\indent Building on this framework, we move from qualitative visual analysis to a quantitative comparison based on feature vectors associated to persistence landscapes. In addition to the discrete landscape feature map (DLFM), we implement and compare several alternative approaches, including continuous (Chen) signatures of landscapes and integrated landscapes, as well as baseline methods based on direct vectorisation and $L_1$ distances.

\indent For the computation of discrete signatures, we use the software FRUITS ({\em Feature Extraction Using Iterated Sums}) \cite{diehl2024fruits}. Continuous path signatures are computed using the \texttt{pysiglib} library \cite{shmelev2025pysiglib}, which provides efficient implementations of Chen signatures for piecewise linear paths. The corresponding code for all experiments in this section is publicly available at \href{https://github.com/dtolosav/discrete-sig-for-landscapes}{https://github.com/dtolosav/discrete-sig-for-landscapes}. We opted for a straightforward implementation of the DLFM, without performing feature selection or dimensionality reduction at the encoding stage.

\indent Note that both path signatures and discrete signatures are amenable to kernelization \cite{kiraly_kernels_2016, lee2025signature}, which has the potential to significantly improve computational aspects while preserving desirable properties.

\paragraph{Correlation between discrete signature and sequence similarity.} 
Using feature representations derived from persistence landscapes, we implement an unsupervised clustering approach and assess its agreement with sequence similarity classes. We use $k$-means clustering with $k=9$, corresponding to the number of structural classes, and evaluate performance using the {\em Adjusted Rand Index} (ARI) and {\em Normalised Mutual Information} (NMI). Proteins in the dataset whose structural class were labelled \textquotedblleft other\textquotedblright \space were not considered in this analysis. To account for the stochastic nature of $k$-means, we repeat the clustering procedure $100$ times with different initialisations and report the average and standard deviation of ARI and NMI across runs. All experiments are performed with fixed random seeds to ensure reproducibility.

\indent We compare several representations: the discrete landscape feature map (DLFM), continuous (Chen) signatures of landscapes and integrated landscapes, flattened landscape vectors (as points in a Euclidean space), and clustering based on $L_1$ distances between landscapes. All methods are evaluated under the same landscape truncation at $15$ levels and using the same clustering procedure, except for the $L_1$-distance method, where no truncation of landscape levels is applied. For the discrete and Chen signature methods, we use signature terms up to weight 3. Details on hyperparameter selection are provided in the accompanying GitHub repository.
\begin{table}[H]
\centering
\begin{tabular}{lcc}
\hline
\textbf{Method} & \textbf{ARI} & \textbf{NMI} \\
\hline
DLFM (discrete signatures) & $\mathbf{0.962 \pm 0.045}$ & $\mathbf{0.882 \pm 0.016}$ \\
Chen signature (landscapes) & $0.708 \pm 0.210$ & $0.778 \pm 0.059$ \\
Chen signature (integrated landscapes) & $0.779 \pm 0.010$ & $0.758 \pm 0.003$ \\
Flattened landscapes & $0.906 \pm 0.041$ & $0.863 \pm 0.017$ \\
$L_1$-distance on landscapes & $0.526 \pm 0.120$ & $0.649 \pm 0.049$ \\
\hline
\end{tabular}
\caption{\footnotesize Comparison of clustering performance (mean $\pm$ standard deviation over $100$ runs of $k$-means) across different feature representations.}
\label{tab:method_comparison}
\end{table}

\indent Table~\ref{tab:method_comparison} summarises the clustering performance across all methods. The DLFM achieves the highest ARI and NMI scores among all representations considered, with ARI $0.962 \pm 0.045$ and NMI $0.882 \pm 0.016$. Notably, this performance is achieved using discrete signatures truncated at weight $3$, indicating that highly informative representations arise already at low weights. We also observe that the variability of DLFM across runs remains moderate, indicating robustness with respect to clustering initialisation.

\indent The DLFM consistently outperforms all alternative approaches. Continuous path signatures of landscapes exhibit significantly lower performance, with high variability across runs, reflecting the loss of discriminative information, likely due to tree-like equivalence. Signatures of integrated landscapes improve stability but remain less effective than DLFM. Direct vectorisation of landscapes performs comparatively well, but still falls short of DLFM, while clustering based on $L_1$ distances performs substantially worse.

\indent It is worth noting that flattened landscape representations achieve relatively strong clustering performance, with ARI and NMI close to those of DLFM. However, this comes at a significant computational cost. In our implementation, the dimension of the flattened landscape vectors is $1{,}680{,}720$, corresponding to the total number of sampled critical points across all proteins and landscape levels. In contrast, the DLFM feature vectors have dimension $8{,}015$, while continuous signature representations have dimension $3{,}616$. Thus, DLFM achieves better performance while using feature vectors that are over two orders of magnitude smaller than those of the flattened landscape approach.

In terms of computational aspects, continuous signatures are currently the fastest method among those considered, thanks to highly optimised implementations. The discrete signature computation is more expensive and typically takes on the order of minutes for the full dataset, while the naive methods based on direct landscape comparison require several hours. All experiments were performed on a machine with
13th Gen Intel Core i9-13900 CPU (32 cores, 5.6 GHz), and 31 Gi RAM, running Ubuntu Linux. Although DLFM is not the fastest method, its stronger clustering performance makes it a promising target for further optimisation. In particular, the present feature vectors are high-dimensional, and we did not perform feature selection or dimensionality reduction at the encoding stage, so there is likely room to substantially reduce the dimension without losing much discriminative power.

\indent Statistical significance of the reported results was checked using a permutation test, obtaining p-values $\leq 0.001$. We further investigated how well the discrete signature captures structural classes by running a {\em centroid-based permutation test with separation ratio}: we compute the {\em average signature} within each structural class, and we compute the {\em separation ratio}, i.e. between-class/within-class distances. This is analogous to the test used in \cite{benjamin2023homology} using average landscapes. We obtain a separation ratio of $2.49$ with a p-value $< 0.001$. This test is independent of the k-means clustering. We refer to the GitHub repository for details on the permutation test and separation ratio.

\indent Overall, these results provide strong empirical evidence that DLFM yields the most informative representation among the methods considered. The fact that DLFM captures sequence similarity, which is a biochemical property, from homological features alone, indicates that the geometric information encoded in the DLFM is biologically meaningful.

\indent We can visualise an instance of clustering using the DLFM in Figure \ref{fig:histogram}: the signature-based clustering separates all structural class representatives except 3ZNC (in orange -- not dominant in any cluster) and 1FUG (in mauve -- dominant in two clusters). However, these exceptions are acceptable and justified by some particular features of the protein dataset. Indeed, the protein class 3ZNC (orange) is broken, presenting missing parts in the amino-acids chain, and it has a similar structure chain to 4QEF (navy blue), as they are both carbamoyltransferases.

\begin{figure}
    \centering
    \includegraphics[width=0.7\linewidth]{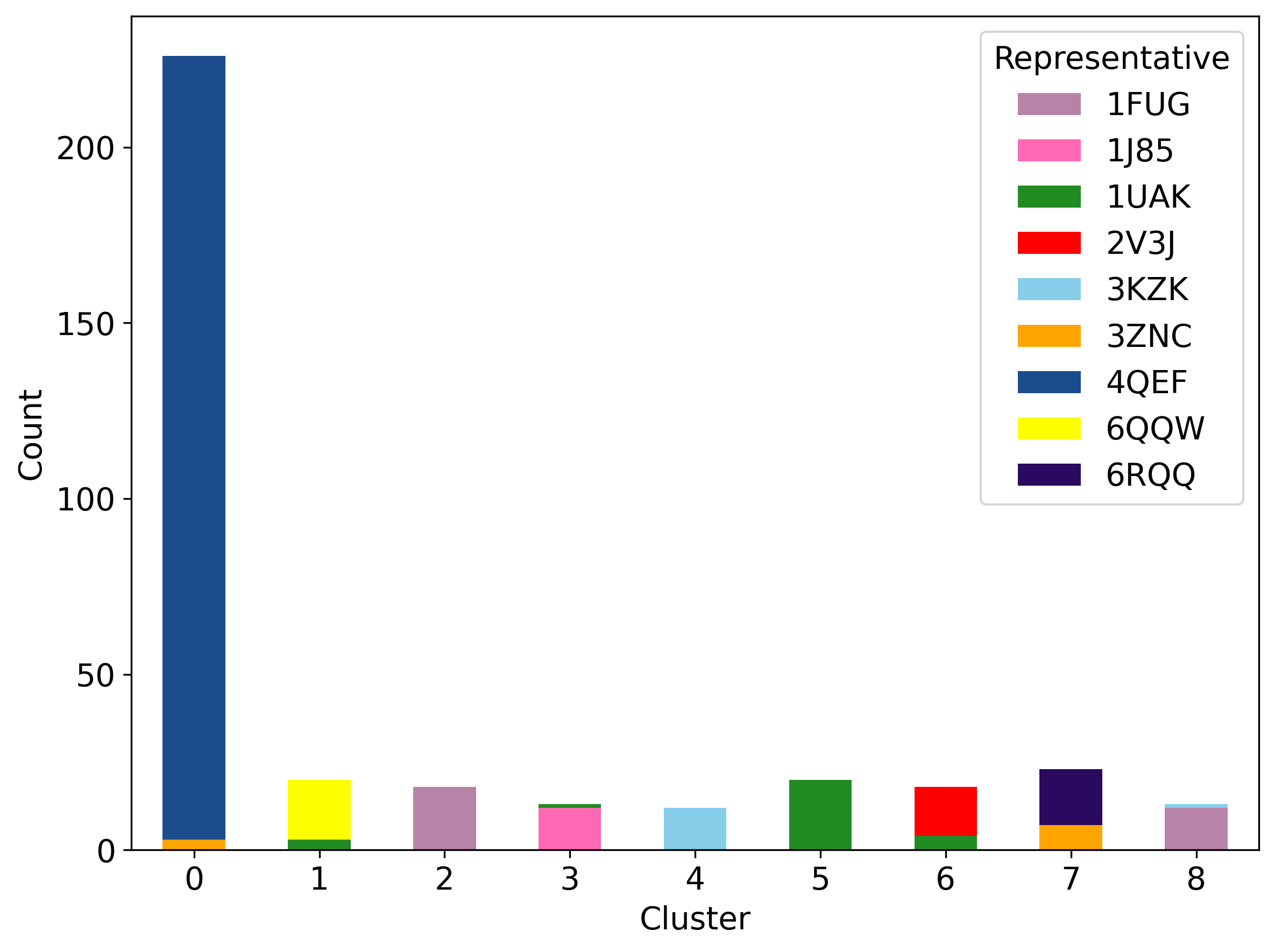}
    \caption{{\footnotesize Example of k-means clustering using DLFM features with $15$ landscape levels, signature weight $\leq 3$, and $9$ clusters. Colors correspond to sequence similarity classes, illustrating strong agreement between the clustering obtained from discrete signatures and the biological labeling.}}
    \label{fig:histogram}
\end{figure}

\paragraph{Discrete signatures for predicting knot depth.}
We now demonstrate that DLFM features can be used more generally as input to standard machine learning methods, beyond simple clustering, by considering a different task: predicting the \textit{knot depth} (\ref{eq:knot_depth}).

\indent We investigate whether DLFM features encode information about knot depth using two complementary approaches. First, we compute the {\em Spearman rank correlation} between pairwise distances in the DLFM feature space and differences in knot depth. This non-parametric measure assesses whether proteins that are close in feature space tend to have similar knot depth, and whether increasing distances in feature space correspond to increasing differences in depth. We obtain a correlation coefficient of $0.647$ with p-value $<0.001$, indicating a statistically significant positive association between discrete signature features and knot depth.

\indent Second, we evaluate the predictive power of DLFM features using a {\em random forest} regression model. Random forest is a learning method based on decision trees, capable of capturing nonlinear relationships between features and target variables \cite{breiman2001random}. Using DLFM feature vectors as input, the model achieves a coefficient of determination $R^2 = 0.866 \pm 0.048$, meaning that approximately $86.6\%$ of the variance in knot depth is explained by the model. Such value of $R^2$ indicates strong predictive performance and shows that the discrete signature provides a sufficiently rich representation to support accurate regression tasks.

\indent These results illustrate that DLFM feature vectors can be used effectively in a variety of machine learning settings, not only for clustering but also for predictive modelling of geometric properties.

\paragraph{Visualisation of DLFM feature space.} The DLFM feature space can be visualised using principal component analysis (PCA). Figure~\ref{fig:PCA 3D} shows the projection onto the first three principal components, colored by sequence similarity (left) and knot depth (right) following the convention used in \cite{benjamin2023homology}: proteins with $D(p)>0.05$ are considerd deeply knotted, those with $D(p)<0.005$ are shallowly knotted, and the remainder are classified as neither. 

While these components explain $70.9\%$ of the total variance, the resulting embeddings still exhibit visible organisation according to both labelings, providing qualitative support for the quantitative results obtained above. As in the Isomap visualisations of \cite{benjamin2023homology}, these projections reveal meaningful structure, while also indicating that the underlying organisation is not fully captured in low-dimensional embeddings. Further visualisations are available in the accompanying repository.

\begin{figure}
    \centering
    \includegraphics[width=1\textwidth]{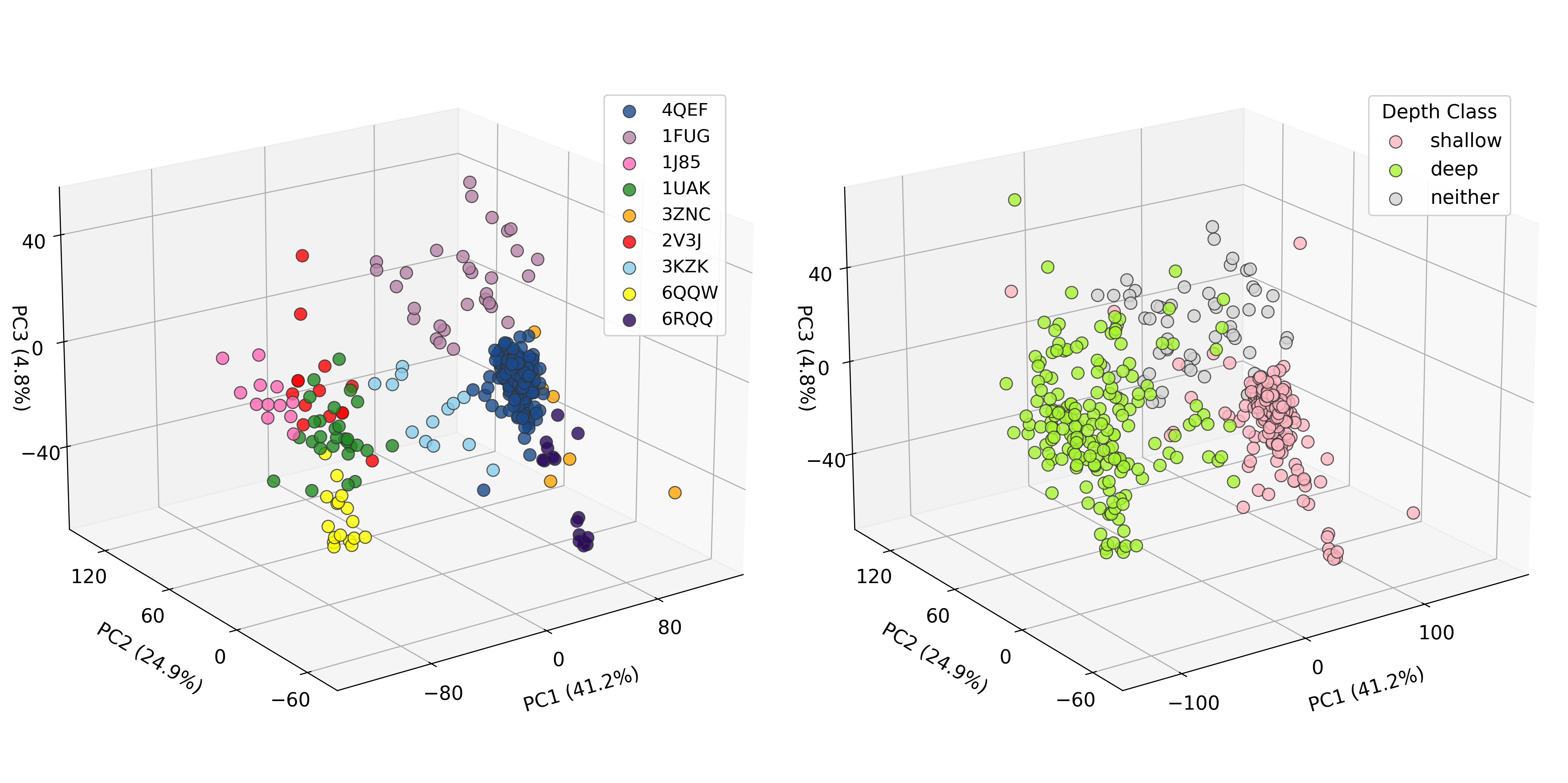}
    \caption{{\footnotesize Three-dimensional PCA projection of DLFM feature vectors. Points are colored by structural class representative (left) and by knot depth classification (right). The embedding reveals visible organisation with respect to both labelings, providing a qualitative illustration of the structure captured by the discrete signature.}}
\label{fig:PCA 3D}
\end{figure}

The experiments above demonstrate that the discrete landscape feature map provides a robust and informative representation of persistence landscapes for downstream data analysis. In particular, DLFM consistently outperforms existing signature-based and landscape-based approaches in recovering biologically meaningful structure, while also supporting predictive modeling of geometric quantities such as knot depth. Together with the theoretical results established earlier, this highlights the potential of discrete signatures as a powerful and flexible tool for topological data analysis in applications.

\section{Appendix: Chen signatures of persistence landscapes}\label{sec:signature for landscapes}

This section is devoted to partial results on the Chen signature tensors of persistence landscape paths from the point of view of algebraic geometry. It would be worth investigating more the geometry of such signature tensors in order to deduce algebraic-geometric descriptions of barcodes. However, this goes beyond our current scope, and we leave it for future work. In algebraic geometry it is standard and often convenient to work over $\mathbb{C}$. Therefore, we consider the signature tensors as elements in the complexification of $\T((V))$, for which we use the same notation.

\paragraph*{Varieties of signature matrices.}  Let $V$ be a $d$-dimensional vector space. For any $k\geq 1$, consider the composition of the path signature with the projection of the tensor algebra on its $k$--th graded component 
\[ \begin{matrix}
    \sigma^{(k)} : & \BV([0,L],V) & \longrightarrow & V^{\otimes k} \\
    & \gamma & \mapsto & \sigma^{(k)}(\gamma)
    \end{matrix}\]
assigning to each path its order--$k$ signature tensor. The {\em $k$--th universal signature variety} $\mathcal{U}_{d,k}$ is the Zariski closure of the (projectivisation) of the image of $\sigma^{(k)}$ inside $\mathbb P(V^{\otimes k})$ (cf.\ \cite[Sec.\ 4.3]{amendola19}). \\
\indent We focus on the case $k=2$, namely {\em varieties of signature matrices}. By construction, these lie in the universal signature varieties $\mathcal{U}_{d,2} \subset \mathbb P(V \otimes V)$. Given a path $\gamma \in \BV([0,L],V)$, the second signature $\sigma^{(2)}(\gamma)\in V^{\otimes 2} \simeq \mathbb C^{d\times d}$ is a matrix with entries
\begin{equation}\label{eq:entry signature matrix}
    \sigma_{ij}^{(2)}(\gamma) = \int_{0}^L \gamma_i(s)\dot{\gamma}_j(s)ds \ . 
    \end{equation}
We denote the symmetric and skew-symmetric part of $\sigma^{(2)}$ by $\sigma^{(2)}_{\rm{sym}}:=\frac{1}{2}(\sigma^{(2)}+(\sigma^{(2)})^{\bfT})$ and $\sigma^{(2)}_{\rm{skew}}:=\frac{1}{2}(\sigma^{(2)}-(\sigma^{(2)})^{\bfT})$ respectively. Applying integration by parts to \eqref{eq:entry signature matrix} leads to 
\begin{equation}\label{eq:sym part}
    2 \sigma^{(2)}_{\rm{sym}} = \gamma(L)^{\otimes 2} \ .
    \end{equation}
In fact, signature matrices are precisely those matrices with rank-$1$ symmetric part \cite[Secc.\ 2,3]{amendola19}, that is
\begin{equation}\label{eq:universal signature matrix}
\mathcal{U}_{d,2} = \left\{ v\otimes v + Q \ | \ v \in V, \ Q \in \textstyle{\bigwedge^2V} \right\} \ . 
\end{equation}

\indent Let $\mathcal{M}_{d,m}$ be the variety of signature matrices of $m$--PWL paths in $V$. Then for any $m$ \cite[Theorem 3.4]{amendola19}
\begin{equation}\label{eq:variety signature matrices}
    \mathcal{M}_{d,m} = \left\{ v \otimes v + Q \ | \ v \in V, \ \rk[v\otimes v | Q]\leq m \right\} \ , 
    \end{equation}
where $[v\otimes v | Q]$ is the matrix obtained by attaching the rank-$1$ matrix $v\otimes v$ and the matrix $Q$. The varieties $\mathcal{M}_{d,m}$ define a filtration starting at the Veronese variety $\nu_{2}(\mathbb P^{d-1})$ and stabilizing at the universal variety of signature matrices \cite[Theorem 3.4, Example 4.12]{amendola19}
\[ \nu_2(\mathbb P^{d-1}) = \mathcal{M}_{d,1} \subset \mathcal{M}_{d,2} \subset \ldots \subset \mathcal{M}_{d,d}=\mathcal{U}_{d,2} \subset \mathbb P(V^{\otimes 2}) \ .\]

\paragraph{Signature tensors of PWL loops.} Since landscapes are PWL loops, it is quite natural to first describe the signature varieties of such larger class of paths. Since signature tensors are invariant under translation, we consider loops starting and ending at $\bold{0}\in \mathbb C^d$. We denote by $\operatorname{log\sigma}(\lambda)$ the {\em log-signature} of $\lambda$ \cite{siegal2024rectifiable}, that is a sequence of tensors defined by the property that $\sigma(\lambda)=\exp(\operatorname{log\sigma}(\lambda))$, and by $\operatorname{log\sigma}^{(k)}(\lambda)$ its $k$-th entry. By definition of the exponential map, the following truncations hold:
{\small \begin{align*}
    \sigma^{(1)}(\lambda) & = \operatorname{log\sigma}^{(1)}(\lambda) \ , \\
    \sigma^{(2)}(\lambda) & = \operatorname{log\sigma}^{(2)}(\lambda) + \frac{1}{2}\left(\operatorname{log\sigma}^{(1)}(\lambda)\right)^2 \ , \\
    \sigma^{(3)}(\lambda) & = \operatorname{log\sigma}^{(3)}(\lambda) + \frac{1}{2}\left(\operatorname{log\sigma}^{(1)}(\lambda)\operatorname{log\sigma}^{(2)}(\lambda) + \operatorname{log\sigma}^{(2)}(\lambda)\operatorname{log\sigma}^{(1)}(\lambda)\right) + \frac{1}{6} \left( \operatorname{log\sigma}^{(1)}(\lambda) \right)^3 \ .
    \end{align*} }

    Each tensor space $(\mathbb{C}^d)^{\otimes k}$ admits a decomposition as direct sum in Thrall modules, see \cite[Sec.\ 3]{amendola2023decomposingtensorspacespath}. Each summand in the right-hand sides of the above equations belongs to a different Thrall module, so all summands are linearly independent.

\begin{proposition}\label{prop:landscape_loops}
    Let $\lambda$ be a PWL path with MSD $\alpha_1 \ast \cdots \ast \alpha_m$ where $\alpha_j(t)=\bold{a}_j t + c_j$ for $\bold{a}_j \in \mathbb C^d$ and suitable $c_j$ in order to make the path conjuction continuous. The following facts are equivalent:
    \begin{enumerate}
        \item $\lambda$ is a loop;
        \item $\sigma^{(1)}(\lambda) = \operatorname{log\sigma}^{(1)}(\lambda) = 0$;
        \item $\sigma^{(2)}(\lambda) = \operatorname{log\sigma}^{(2)}(\lambda) = \frac{1}{2}\sum\limits_{i<j} \bold{a}_i \wedge \bold{a}_j$;
        \item $\sigma^{(3)}(\lambda) = \operatorname{log\sigma}^{(3)}(\lambda)$.
    \end{enumerate}   
\end{proposition}
\begin{proof}
In the above notation, the first signature of the $m$-PWL path $\lambda$ is $\sigma^{(1)}(\lambda)=\bold{a}_1 + \ldots + \bold{a}_m$. The latter sum being the zero vector is equivalent to $\lambda$ being a loop. Moreover, the first signature always coincides with the first log-signature, hence $(1)\iff (2)$. The second equality in $(3)$ always holds (cf.\ \cite[Example 7.3]{siegal2024rectifiable}). If the first log-signature is zero, then $\sigma^{(2)}(\lambda)=\operatorname{log\sigma}^{(2)}(\lambda)$. On the contrary, if the first equality in $(3)$ holds, then $(\operatorname{log\sigma}^{(1)}(\lambda))^{\otimes 2} =\bold{0}$ (as matrices) implies $\operatorname{log\sigma}^{(1)}(\lambda)=\bold{0}$ (as vectors). This proves $(2)\iff (3)$. Finally, the equivalence $(2) \iff (4)$ follows from the linear independence among summands lying in different Thrall modules.
\end{proof}

\indent The above characterisations of PWL loops in terms of their signature matrices allow to describe the corresponding varieties. We denote the variety of signature matrices of $m$--PWL loops in $\mathbb C^d$ by 
\[\mathcal{M}_{d,m}^{\circlearrowright} := \overline{\left\{ \sigma^{(2)}(\lambda) \ | \ \lambda \ \text{$m$-PWL loop in} \ \mathbb C^d \right\}} \subset \mathcal{M}_{d,m} \subset \mathbb P(\mathbb C^{d}\otimes \mathbb C^d) \ , \]
and the variety of signature matrices of loops in $\mathbb C^d$ (not necessarily PWL) by 
\[ \mathcal{U}^{\circlearrowright}_{d,2}:= \overline{\left\{ \sigma^{(2)}(\lambda) \ | \ \lambda \ \text{loop in} \ \mathbb C^d \right\}} \subset \mathcal{U}_{d,2} \ . \]
Let $Gr(2,d)\subset \mathbb P(\bigwedge^2\mathbb C^d)$ the Grassmannian of planes in $\mathbb C^d$, corresponding to $d\times d$ skew-symmetric matrices of rank $2$, and let $\sigma_r(Gr(2,d))\subset \mathbb P(\bigwedge^2\mathbb C^d)$ be the $r$--th secant variety of $Gr(2,d)$, corresponding to $d\times d$ skew-symmetric matrices of rank at most $2r$.  

\begin{corollary}\label{cor:loop_variety}
In the above notation, it holds:
\begin{enumerate}
\item $\mathcal{U}^{\circlearrowright}_{d,2}= \mathbb P(\bigwedge^2V)$;
\item $\mathcal{M}_{d,m}^{\circlearrowright} = \sigma_{\lfloor \frac{m}{2}\rfloor}(Gr(2,d))$, for any $m\leq d$.
\end{enumerate}
In particular, $\mathcal{M}_{d,d+h}^{\circlearrowright}=\mathcal{U}^{\circlearrowright}_{d,2}=\mathbb P(\bigwedge^2V)$ for any $h\geq 0$.
\end{corollary}
\begin{proof}
    Up to translation, a loop $\lambda:[0,L]\rightarrow \mathbb C^d$ is such that $\lambda(0)=\lambda(L)=\bold{0}$, hence the signature matrix of $\lambda$ has symmetric part $\sigma^{(2)}_{\rm{sym}}(\lambda)=\frac{1}{2}(\lambda(L)-\lambda(0))^{\otimes 2}=\bold{0}^{\otimes 2}$. This also follows from Proposition \ref{prop:landscape_loops}. Then the two theses follow straightforward from the definitions \eqref{eq:universal signature matrix} and \eqref{eq:variety signature matrices} by cutting the varieties $\mathcal{M}_{d,m}\subset \mathcal{U}_{d,2}\subset \mathbb P(V^{\otimes 2})$ by $\mathbb P(\bigwedge^{2}V)$.
\end{proof}

\paragraph{Signature tensors of landscapes.} By construction, the number $m$ of segments in the MSD of a landscape with $d$ levels (i.e. a landscape path in $\mathbb R^d$) is {\em strictly lower bounded} by $d$, that is $d \lneq m$. Thus it makes sense to consider the (universal) variety of signature matrices of landscapes in $\mathbb R^d$ regardless of the number of linear pieces in the MSD. We denote such variety by 
\[ \mathcal{M}_{d,2}^{\LS} := \overline{\left\{ \sigma^{(2)}(\lambda) \ | \ \lambda \ \text{landscape path in} \ \mathbb R^d \right\}} \subset \mathcal{M}_{d,2}^\circlearrowright=\mathbb P\left(\textstyle{\bigwedge^2}\mathbb C^d\right) \ . \]
\indent Consider a landscape path $\lambda^B(t)=(\lambda_1(t),\ldots, \lambda_d(t))\subset \mathbb R^d$ with MSD $\lambda^B=\alpha_1\ast \cdots \ast \alpha_m$ defined on the partition $[0,t_1]\cup [t_1,t_2] \cup \ldots \cup [t_{m-1},L]$ such that each segment $\alpha_j:[t_{j-1},t_j]\rightarrow \mathbb R^d$ is defined as $\alpha_j(t)=\bold{a}_j\cdot t + \alpha_{j-1}(t_{j-1})$ for a suitable $\bold{a}_j=(a_{1j} , \ldots , a_{dj})^{\bfT}\in \{-1,0,1\}^d$. Then each $\lambda^B$ defines a $d\times m$ matrix $X_\lambda := [\bold{a}_1 | \cdots | \bold{a}_m]$.\\
\indent The conditions on the levels of a landscape (cf.\ Sec.\ \ref{subsec:preliminaries landscapes}) impose some semialgebraic conditions, besides the algebraic condition of being a loop: clearly, $\mathcal{M}_{d,2}^{\LS} \subseteq \mathcal{M}_{d,2}^\circlearrowright =\mathbb P\left(\bigwedge^2\mathbb C^d \right)$. For instance, the non-negativity of sums over each row implies $\bold{a}_1+\ldots + \bold{a}_k\geq 0$ for any $k\leq m$, while the dominance of higher levels is equivalent to $\sum_{j=1}^k a_{ij} \geq \sum_{j=1}^k a_{i+1 \, j}$. But there are also non semi-algebraic relations, such as the one on the modulus of the entries. We leave the study of the variety $\mathcal{M}_{d,2}^{\LS}$ for future work. Observe that the orbit argument in \cite[Example 2.3]{amendola19} does not apply since the number of segments $m$ is greater than $d$. \\
\textcolor{black}{\indent Observe that the space of $d$-barcodes depends on $2d-1$ parameters, as one can choose $2d$ values for the $d$ birth-death points but initialise the first birth at $0$. By the definition of levels in \eqref{eq:def level}, it follows that a barcode of $d$ intervals defines a landscape path with at most $d$ levels. Therefore, via the landscape embedding, a barcode with $d$ intervals defines a (skewsymmetric) signature matrix of size $d\times d$. Assuming that such mapping is injective, then the variety $\mathcal{M}_{d,2}^{\LS}\subseteq \mathbb P(\bigwedge^2\mathbb C^d)$ has dimension at most $2d-1$. A dimensional count implies that the variety $\mathcal{M}_{d,2}^{\LS}$ does not fill the ambient space for every $d\geq 6$. It would be interesting to determine such variety.}

\begin{conjecture}
For $d\leq 5$ the variety of signature matrices of landscapes $\mathcal{M}_{d,2}^{\LS}$ coincides with the variety of signature matrices of loops $\mathcal{M}_{d,2}^{\circlearrowright}=\mathbb P(\bigwedge^2\mathbb C^d)$.
\end{conjecture}

{\scriptsize
\printbibliography
}

\end{document}